\documentclass[a4paper,reqno]{amsart}

\usepackage{amsmath}
\usepackage{enumerate}
\usepackage{mathrsfs}
\usepackage[all]{xy} 
\usepackage{tikz}

\newtheorem{Thm}{Theorem}[section]\newtheorem*{Thm*}{Theorem}
\newtheorem{Lem}[Thm]{Lemma} 
\newtheorem{Cor}[Thm]{Corollary}

\newtheorem{Prop}[Thm]{Proposition}
\newtheorem{Prop-Def}[Thm]{Proposition-Definition}

\theoremstyle{definition}

\newtheorem{Ex}[Thm]{Example}
\newtheorem{Def}[Thm]{Definition}

\newtheorem{Rem}[Thm]{Remark}

\newcommand{\0}{\hspace{-3pt}0}

\newcommand{\ra}{\rightarrow}
\newcommand{\la}{\leftarrow}
\newcommand{\D}{\mathsf{D}}
\newcommand{\K}{\mathsf{K}}
\newcommand{\SSS}{\mathbb{S}}

\newcommand{\cal}{\mathcal}
\newcommand{\T}{\mathcal T}
\newcommand{\sss}{\mathcal S}
\newcommand{\X}{\mathcal X}
\newcommand{\rrr}{\mathcal R}

\newcommand{\zzz}{\mathcal Z}
\newcommand{\Y}{\mathcal Y}
\newcommand{\ZZ}{\mathcal Z}

\newcommand{\hhh}{\mathcal H}

\newcommand{\p}{\mathrm{p}}
\newcommand{\SMC}{\mathrm{SMC}}

\newcommand{\bb}{\mathrm{b}}

\newcommand{\Filt}{\mathsf{Filt}}

\newcommand{\con}{{\rm cone}}

\newcommand{\h}{{\mathrm H}}
\renewcommand{\dim}{{\rm dim}}

\newcommand{\xra}{\xrightarrow}
\newcommand{\Z}{{\mathbb Z}}

\newcommand{\lan}{\langle}
\newcommand{\ran}{\rangle}
\newcommand{\un}{\underline}
\newcommand{\op}{\oplus}

\newcommand{\ot}{\otimes}

\newcommand{\hs}{\hspace{-3pt}}

\newcommand{\sg}{\operatorname{sg}\nolimits}

\newcommand{\Hom}{\operatorname{Hom}\nolimits}
\newcommand{\rad}{\operatorname{rad}\nolimits}
\newcommand{\Top}{\operatorname{top}\nolimits}

\newcommand{\im}{\operatorname{Im}\nolimits}

\newcommand{\End}{\operatorname{End}\nolimits}
\newcommand{\RHom}{\mathbf{R}\strut\kern-.2em\operatorname{Hom}\nolimits}
\newcommand{\RshHom}{\mathbf{R}\strut\kern-.2em\mathscr{H}\strut\kern-.3em\operatorname{om}\nolimits}
\newcommand{\shHom}{\mathscr{H}\strut\kern-.3em\operatorname{om}\nolimits}
\newcommand{\shEnd}{\mathscr{E}\strut\kern-.3em\operatorname{nd}\nolimits}

\DeclareMathOperator{\moduleCategory}{{\mathsf{mod}}} \renewcommand{\mod}{\moduleCategory}
\DeclareMathOperator{\proj}{\mathsf {proj}}
\DeclareMathOperator{\thick}{\mathsf{thick}}
\DeclareMathOperator{\per}{\mathsf{per}}
\DeclareMathOperator{\add}{\mathsf {add}}
\DeclareMathOperator{\CM}{\mathsf {CM}}

\DeclareMathOperator{\Proj}{\mathsf{Proj}}

\numberwithin{equation}{section}

\begin{document}


\title[SMC reduction of triangulated categories]{Reductions of triangulated categories \\ and simple-minded collections}
\author{Haibo Jin}
\address{Haibo Jin: Graduate School of Mathematics, Nagoya University, Furocho, Chikusaku, Nagoya 464-8602, Japan}
\email{d16002n@math.nagoya-u.ac.jp}

\begin{abstract}
Silting  and Calabi-Yau reductions are   important process in representation theory to construct new triangulated categories from given one, which are similar to Verdier quotient.
 In this paper, first
we  introduce a new reduction process of triangulated category, which is analogous to the silting (Calabi-Yau) reduction. 
For a triangulated category $\T$ with a pre-simple-minded collection (=pre-SMC) $\cal R$, we construct a new triangulated category $\cal U$ such that the SMCs in $\cal U$ bijectively correspond to those in $\T$ containing $\cal R$.
Secondly, we give an analogue of Buchweitz's theorem for  the singularity category $\T_{\sg}$ of a SMC quadruple $(\T,\T^{\p},\SSS, \cal S)$: the category $\T_{\sg}$ can be realized as the stable category of  an extriangulated subcategory $\cal F$ of $\T$.
Finally, we show the SMS (simple-minded system) reduction due to Coelho Sim\~oes and  Pauksztello is the shadow of  our SMC reduction. This is parallel to the result that Calabi-Yau reduction is the shadow of  silting reduction  due to Iyama and Yang.
\end{abstract}
\maketitle

\tableofcontents

\section{Introduction}

Triangulated categories appear in many branches of mathematics, such as algebraic geometry, representation theory and algebraic topology.
In derived categories, there are two important classes of objects: projective objects and simple objects. Projective objects (or more generally, tilting objects) play a central role in tilting theory, which is  one of the standard tools for studying triangulated categories. Their variants,
silting objects and cluster tilting objects, have been
used to study positive Calabi-Yau (= CY) triangulated categories \cite{BMRRT, IY1, KR, KMV} and  the categorification of cluster category  \cite{FZ}. 
On the other hand, simple objects, or more generally, simple-minded collections (=SMCs) are also well-studied   in derived categories. They are   important  in Koszul duality
\cite{BGS, KN}, and bijectively correspond to silting objects \cite{R, KY}.
 Simple-minded systems (= SMSs) in stable module categories were introduced in \cite{KL} and studied  for negative CY triangulated categories in \cite{D, CS3}. 
Recently, there is increasing interest in negative CY triangulated categories (see, for example \cite{CS1, CS2, CS3, CSP}), including the stable categories of Cohen-Macaulay (= CM) dg modules \cite{J}.

There are two useful tools  to study the class of  silting (resp. cluster-tilting, SMC, SMS) objects in a triangulated category $\T$. One is mutation, which gives a new object in this class from a given one. Another is  reduction, which is a new triangulated category $\cal U$ realized as a  certain sub (or subfactor) category of  $\T$.  
There is a bijection between silting (resp. cluster-tilting, SMS, SMC) objects in $\cal U$ and those in $\T$ with some properties. The following  picture shows some works on these subjects, where the reduction of SMC was not studied before. 


{\small
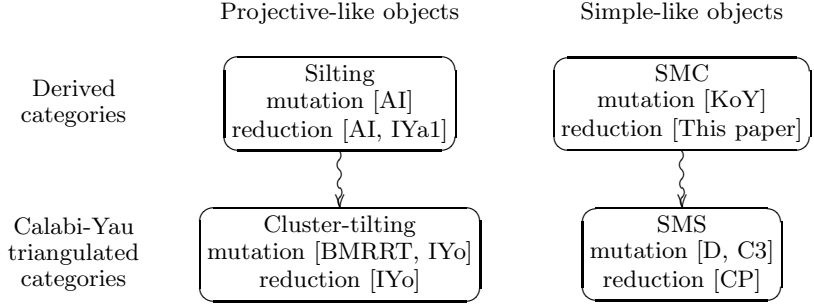
\[\begin{xy}
(30,32)*{\txt{Projective-like objects}},
(-5,20)*\txt{Derived\\ categories},
(30,20)*+[F:<5pt>]{\txt{Silting\\
mutation \cite{AI} \\
reduction \cite{AI,IY2}}}="1",
(-5,0)*{\txt{Calabi-Yau \\ triangulated\\ categories}},
(30,0)*+[F:<5pt>]{\txt{Cluster-tilting\\
mutation \cite{BMRRT,IY1}\\
reduction \cite{IY1}}}="2",
(75,32)*{\txt{Simple-like objects}},
(75,20)*+[F:<5pt>]{\txt{SMC\\
mutation \cite{KY}\\
reduction [This paper]}}="3",
(75,0)*+[F:<5pt>]{\txt{SMS\\
mutation \cite{D,CS3}\\
reduction \cite{CSP}}}="4",
\ar@{~>}"1";"2" 
\ar@{~>}"3";"4" 
\end{xy}\]}

Thus our first aim of this paper is to introduce the \emph{SMC reduction}.
 For   a pre-SMC  $\cal R$ (which is a SMC without generating condition) of a Krull-Schmidt triangulated category $\T$, the corresponding SMC reduction is  the Verdier quotient $\cal U=\T/\thick(\cal R)$.  
 One can realize $\cal U$ as the  
 additive subcategory $$\cal Z=\cal R[\ge\0]^{\perp}\cap {}^{\perp}\cal R[\le\0]$$
  of $\T$
   under certain assumptions (R1) and (R2) in Section \ref{Section:SMCreduction}.
  Namely,
  
\begin{Thm}[Theorem \ref{Thm:SMCbij}]\label{Thm:1.1}
Under the setting above, the following results hold.
 \begin{enumerate}[\rm (1)]
   \item  The composition $\cal Z \hookrightarrow \T \ra \cal U$ gives an equivalence $\cal Z \xra{\simeq} \cal U$;
      \item There is a bijection  
   \[ {\rm SMC}_{\rrr} \T:=\{ \text{SMCs in $\T$ contain $\cal R$}\}   \longleftrightarrow {\rm SMC} \,\cal U:= \{ \text{SMCs in $\cal U$}\}.\]
 \end{enumerate}
 \end{Thm}
  Since $\zzz$ is not closed under $[\pm 1]$, it dose not have  a triangulated structure a priori. Nevertheless, the theorem above shows that it has a canonical triangulated structure induced by $\cal U$.
Also notice that,  Theorem \ref{Thm:1.1} can be regarded as a dual of   silting reduction \cite{IY2}, where it was necessary to take an ideal quotient of $\cal Z$. In \cite{J},  Theorem \ref{Thm:1.1} was used to construct SMCs and it  played an important role in the proof of \cite[Theorem 7.1]{J}. 

The second aim of this paper is to generalize the singularity category of a finite dimensional  Gorenstein  $k$-algebra  $A$ over a field $k$.
In this case, the singularity category is defined as  the Verdier quotient $\D_{\sg}(A)=\D^{\bb}(\mod A)/\K^{\bb}(\proj A)$  by \cite{B, O}.
  Buchweitz's equivalence  states that $\D_{\sg}(A)$ is triangle equivalence to the stable category  $\un{\CM}A$ of Cohen-Macaulay $A$-modules.  A key observation in our context is that 
  $\D^{\bb}(\mod A)$ has a SMC  consisting of simple $A$-modules,  
   and  there is a  relative Serre functor $\nu=?\ot_{A}^{\bf L}DA$.
  
  To generalize the notion of singularity categories  and   Buchweitz's equivalence, we work on 
 a \emph{SMC quadruple}  $(\T, \T^{\rm p}, \SSS,\cal S)$, where $\T^{\p}$ is a thick subcategory of a triangulated category $\T$, $\SSS$ is a relative Serre functor, $\sss$ is a SMC of $\T$ and they satisfy some conditions (see Definition \ref{Def:relativeSerre}). 
We define the \emph{singularity category} as the Verdier quotient  $$\T_{\sg}:=\T/\T^{\p}.$$ 
In this setting, we have a co-$t$-structure $\T={}\T_{>0}\perp \T_{\le 0}$, where $\T_{>0}={}^{\perp}\sss[\ge\0]$ and $\T_{\le 0}={}^{\perp}\sss[<\0]$.
 Using them we define subcategories
\[
\cal F = \T_{>0}^{\perp}\cap {}^{\perp}(\T_{\le -1}\cap \T^{\p}),  \
\cal P= \T_{\ge 0}\cap \T_{\le 0}, \]
where in the algebra case above, $\cal F=\CM A$ and $\cal P=\proj A$.
Our second result realizes  $\T_{\sg}$ as a subfactor category of $\T$.

 \begin{Thm}[Theorem \ref{Thm:singularity} (1), (2)]\label{Thm:i2}
 Let $(\T, \T^{\rm p}, \SSS,\cal S)$ be a SMC quadruple and let $\T_{\sg}$, $\cal F$, $\cal P$  be defined as above.  Then 
\begin{enumerate}[\rm (1)]
\item $\cal F$ is a Frobenius extriangulated category with $\Proj \cal F=\cal P$ \rm{(}in the sense of \cite{NP}\rm{)};
\item
The composition
\[ \cal F \subset \T \ra \T_{\sg}  \]
induces an equivalence $\pi: \frac{\cal F}{[\cal P]} \xra{\simeq} \cal \T_{\sg}$. Moreover,  $\T_{\sg}$ has a Serre functor $\SSS[-1]$.
\end{enumerate}
 \end{Thm}

Theorem \ref{Thm:singularity} can be regard as a dual of the equivalence between the fundamental domain and the cluster category \cite{Am, G, IY2}, where it was not necessary to  take ideal  quotient.

An important case of Serre quadruple is 
non-positive  \emph{CY triple}, which is a Serre quadruple
$(\T, \T^{\p}, \SSS, \cal S)$  with $\SSS=[-d]$ for $d\ge 0$. In the rest  part of introduction,  we will  focus on $(-d)$-CY triple.  In this case, there is a nice description of $\cal F$ as follows.

\begin{Prop}[Theorem \ref{Thm:singularity} (3)]\label{Prop:3}
Let $(\T, \T^{\p}, \sss)$ be a $(-d)$-CY triple. Then  $\cal F=\cal H[d]\ast\cal H[d-1]\ast\cdots\ast\cal H$ and $\sss$ is a $d$-SMS in $\T_{\sg}$, where $\cal H=\Filt\cal S$ is  the extension-closed subcategory generated by $\cal S$.
\end{Prop}

 A typical example of Theorem \ref{Thm:i2} and Proposition \ref{Prop:3} was considered in \cite{J}, where proper Gorenstein dg $k$-algebras and their Cohen-Macaulay modules were studied.

 The third aim of this paper is to connect our SMC reductions  and the SMS reductions defined by Coelho Sim\~oes and  Pauksztello \cite{CSP}.    We first show that the  SMC reduction of a CY triple gives rise to a new CY triple.
 
 \begin{Thm}[Theorem \ref{Thm:SMCred}] \label{Thm:i3}
 Let $(\T, \T^{\p}, \cal S)$ be a $(-d)$-CY triple. Let $\cal R$ be a subset of $\cal S$ such that the extension-closed subcategory $\hhh_{\rrr}$ generated by $\rrr$ is functorially finite in $\T$ . Let $\cal U$ be the SMC reduction of $\cal T$ with respect to $\cal R$. 
 Then the triple $(\cal U, \cal U^{\p}, \cal S)$  is also a  $(-d)$-CY triple, where one can regard $\cal U^{\p}:= \T^{\p}\cap (\thick\cal R)^{\perp}$ as  a subcategory of $\cal U$.
 \end{Thm}
 
 For a $(-d)$-CY triple $(\T, \T^{\p}, \cal S)$, we know $\T_{\sg}$ is a $(-d-1)$-CY triangulated category
 by Theorem \ref{Thm:i2} (2),  and we can consider the SMS reduction $(\T_{\sg})_{\cal R}$ in $\T_{\sg}$ with respect to $\cal R$ introduced in  \cite{CSP}. Our main theorem of this paper  shows that SMS reduction is the shadow of  SMC reduction in the following sense.
 
 \begin{Thm}[Theorem \ref{Thm:mainresult}]
 Keep the assumption in Theorem $\ref{Thm:i3}$.
Then there is a triangle equivalence from the singularity category  $\cal U_{\sg}$ to the SMS reduction  $(\cal T_{\sg})_{\cal R}$ of the singularity category $\T_{\sg}$ with respect to $\cal R$.
 \end{Thm}

 This can be illustrated by the following commutative diagram of operations.
 
 \[\xymatrixcolsep{6pc}\xymatrixrowsep{4pc}\xymatrix{   \cal T \ar[r]^{\text{sing.   category}}  \ar[d]_{\text{SMC reduction}}& \T_{\sg} \ar@{~>}[d]^{\text{SMS reduction}}\\ \cal U \ar[r]^{\text{sing. category}} & \cal U_{\sg} \cong (\T_{\sg})_{\cal R}
}\]

 The diagram above induces a commutative diagram of maps
 
 \[\xymatrixcolsep{6pc}\xymatrixrowsep{4pc}\xymatrix{ \text{SMCs in $\T$  contains $\cal R$} \ar[r]  \ar[d]& \text{SMSs in $\T_{\sg}$ contains $\cal R$}\ar[d]\\
 \text{SMCs in $\cal U$} \ar[r]  & \text{SMSs in $(\T_{\sg})_{\cal R}$} }\]
 where the horizontal two maps above are well-defined  under mild conditions (see Theorem \ref{Thm:S'}).
 The results we obtain here are parallel to the connection between  silting reductions and CY reductions given in \cite{IY2}. 

In Appendix \ref{appendix}, we give a triangle equivalence induced by derived Schur functors. It provides us an important class of examples on SMC reduction and it is also useful itself.

\medskip\noindent
{\bf Acknowledgements }
The author is supported by China Scholarship Council. He would like to thank  Prof. Osamu Iyama for many useful discussions. 

\section{Preliminaries}

\subsection{Notation}

Throughout this paper, $k$ denotes a field.
Let $\T$ be an additive category. Let $\cal S$ be a full subcategory of $\T$. For an object $X$ in $\T$, a morphism $f: S\ra X$ is called a \emph{right $\cal S$-approximation} if   $S\in \cal S$ and $\Hom_{\T}(S',f)$ is surjective for any $S'\in \cal S$.
We say $\cal S$ is \emph{contravariantly finite} if every object in $\T$ has a right $\cal S$-approximation. Dually, we define \emph{left $\cal S$-approximation} and \emph{covariantly finite} subcategories. We say $\cal S$ is \emph{functorially finite} if it is both contravariantly finite and covariantly finite. 

We denote by $\add \cal S$ the smallest full subcategory of $\T$ containing $\cal S$ and closed under isomorphism, finite direct sums, and direct summands. Denote by $[\cal S]$ the ideal of $\T$ consisting of morphisms which factor through an object in $\add \cal S$ and denote by $\frac{\cal T}{[\cal S]}$ the additive quotient of $\cal T$ by $\cal S$. Define subcategories
\begin{eqnarray*}
{}^{\perp}\cal S &:=&\{ X\in \T\mid \Hom_{\T}(X, \cal S)=0\}, \\
\cal S^{\perp}&:=&\{ X\in \T\mid \Hom_{\T}(\cal S, X)=0\}. 
\end{eqnarray*}

 We denote by $[1]$ (or $\lan 1\ran$) the suspension functors for  triangulated categories. Let $\cal T$ be a triangulated category. For any $X, Y\in \T$ and $n\in \Z$, when we write $\Hom_{\T}(X, Y[>\hspace{-3pt}n])=0$ (resp. $\Hom_{\T}(X, Y[<\hspace{-3pt}n])=0$, $\Hom_{\T}(X, Y[\ge\hspace{-3pt} n])=0$,   $\Hom_{\T}(X, Y[\le \hspace{-3pt}n])=0$), we mean $\Hom_{\T}(X, Y[i])=0$ for all $i>n$ (resp. $i<n$, $i\ge n$, $i\le n$).

 Let $\cal S$ be a full subcategory of $\T$. We denote by $\thick (\cal S)$ the smallest thick subcategory containing $\cal S$. Let $\cal S'$ be another full subcategory of $\T$. Define a new subcategory of $\T$ as follows. 
 \begin{eqnarray*}
 \cal S \ast \cal S' &:= & \{ X\in \cal T \mid \text{there is a triangle $S\ra X \ra S' \ra S[1]$ with $S\in \cal S$ and $S'\in\cal S'$} \}.
 \end{eqnarray*}
 If $\Hom_{\T}(\cal S, \cal S')=0$, that is, if $\Hom_{\T}(S, S')=0$ for any $S\in \cal S$ and $S'\in \cal S'$, we write $\cal S\ast \cal S'=\cal S\perp \cal S'$. 
 For subcategory $\cal S_{1}, \cdots, \cal S_{n}$ of $\T$, we define $\cal S_{1}\ast\cdots\ast\cal S_{n}$ and $\cal S_{1}\perp\cdots\perp \cal S_{n}$ inductively.
We say $\cal S$ is \emph{extension-closed} if $\cal S\ast\cal S=\cal S$. We denote by $\Filt(\cal S)$ the smallest extension-closed subcategory of $\T$ containing $\cal S$.
It is easy to see $\Filt(\cal S)=\bigcup_{n\ge1}\underbrace{\cal S\ast\cdots\ast \cal S}_{n}$. 
We write $\Filt(\cal S[\ge \hspace{-3pt}n])=\Filt(\bigcup_{i\ge n}\cal S[i])$ and  $\Filt(\cal S[\le \hspace{-3pt}n])=\Filt(\bigcup_{i\le n}\cal S[-i])$.

Here we recall some well-known results on additive closures and approximations for later use.
\begin{Lem}\label{Lem:notation}
Let $\T$ be a Krull-Schmidt triangulated category. Let $\X$ and $\Y$ be two extension-closed subcategories of $\T$. Then 
\begin{enumerate}[\rm(1)]
\item If $\Y\ast\X\subset\X\ast\Y$, then $\X\ast\Y$ is also extension-closed;
 \item If $\Hom_{T}(\X, \Y)=0$, then $\add(\X\ast\Y)=\X\ast\Y$; 
 \item Let $T\in \T$. Let $R_{T}\xra{f} T\ra T' \ra X[1]$ be the triangle extended  by  the minimal right $\X$-approximation  $f$ of $T$. Then $T' \in \X^{\perp}$. If moreover, $\Hom_{\T}(\Y,T)=0$, then $f$ is also a minimal right $(\Y\ast\X)$-approximation of $T$. 
\end{enumerate} 
\end{Lem}

\begin{proof}
(1)  follows from $(\X\ast\Y)\ast(\X\ast\Y)=\X\ast(\Y\ast\X)\ast\Y\subset \X\ast\X\ast\Y\ast\Y=\X\ast\Y$.

(2) See  \cite[Proposition 2.1 (1)]{IY1}.

(3) The first assertion follows from the proof of \cite[Proposition 2.3 (1)]{IY1} and the second one is easy to check.
\end{proof}

\subsection{$t$-structure and co-$t$-structures}
\label{Section:tstr}

Let $\cal T$ be a triangulated category. Let $\cal X$ and $\cal Y$ be two full subcategories of $\cal T$. If $\cal T=\cal X\perp \cal Y$, $\cal X^{\perp}=\cal Y$ and ${}^{\perp}\cal Y=\cal X$ hold, we say $\cal T=\cal X\perp \cal Y$ is a \emph{torsion pair}. 
If a torsion pair $\cal T=\cal X\perp\cal Y$ satisfies $\cal X[1]\subset \cal X$ (resp. $\cal Y[1]\subset \cal Y$), we call it a \emph{$t$-structure} (resp. \emph{co-$t$-structure}), in this case, we denote by $\cal H=\cal X\cap\cal Y[1]$ (resp. $\cal P=\cal X\cap \cal Y[-1]$) the \emph{heart} (resp. \emph{co-heart}). We say a $t$-structure $\cal T=\cal X\perp \cal Y$ is \emph{stable} if $\cal X[1]=\cal X$.

Let $\cal S$ be a thick subcategory of $\T$. Let  us recall  a sufficient condition for the  Verdier quotient $\cal T/\cal S$ to be realized as an ideal quotient given in \cite{IY3}. We consider the following setting.

\begin{itemize}
\item[(T0)] $\cal T$ is a triangulated category, $\cal S$ is a thick subcategory of $\cal T$ and $\cal U=\T/\cal S$;
\item[(T1)]  $\cal S$ has a torsion pair $\cal S=\cal X\perp\cal Y$;
\item[(T2)] $\cal T$ has torsion pairs $\cal T=\cal X\perp \cal X^{\perp}={}^{\perp}\cal Y\perp\cal Y$. 
\end{itemize}
Let $\cal Z:= \cal X^{\perp}\cap {}^{\perp}\cal Y[1]$ and $\cal P:=\cal X[1]\cap \cal Y$.   Then 

\begin{Prop}\cite[Theorem 1.1]{IY3}\label{Thm:IY}
Under the assumptions (T0), (T1) and (T2), the composition $\cal Z\subset \T \ra \cal U$ induces an equivalence of additive category $\frac{\cal Z}{[\cal P]}\cong \cal U$. In particular, the category $\frac{\cal Z}{[\cal P]}$ has a structure of a triangulated category.
\end{Prop}

\begin{Rem}\label{Rem:extri}
If (T0), (T1) and (T2) hold, we may regard $\cal Z$ as a Frobenius  extriangulated category with   $\Proj \cal Z=\cal P$  in the sense of \cite{NP} (see \cite[Section 1.2]{IY3}).
 \end{Rem}

\subsection{Simple-minded collections and simple-minded systems}

Let $\cal T$ be a Krull-Schmidt triangulated category and let $\cal S$ be a subcategory of $\T$.

\begin{Def}\label{Def:SMC}
We call $\cal S$ a \emph{pre-simple-minded collection} (\emph{pre-SMC}) if for any $X, Y\in \cal S$, the following conditions hold.
\begin{enumerate}[\rm (1)]
 \item $\Hom_{\T}(X, Y[<\0])=0$;
 \item $\dim_{k}\Hom_{\T}(X, Y)=\delta_{X, Y}$.
\end{enumerate} 
We call $\cal S$ a \emph{simple-minded collection} (\emph{SMC}) if $\cal S$ is a  pre-SMC and moreover,  $\thick(\cal S)=T$.
\end{Def}

For any pre-SMC, there is a standard $t$-structure  
associated to it in the following sense, see \cite[Corollary 3 and Proposition 4]{A2} or \cite[Proposition 5.4]{KY}. 
\begin{Prop}\label{Prop:SMCtotstr}
Let $\cal R$ be a pre-SMC of $\cal T$. Let $\cal X_{\cal R}:=\Filt(\cal R[\ge \0])$ and  $\cal Y_{\cal R}:= \Filt(\cal R[< \0])$.  $\hhh_{\rrr}=\Filt(R)$.  Then 
\begin{enumerate}[\rm (1)]
\item
We have a  bounded $t$-structure 
$\T=\cal X_{\rrr}\perp \cal Y_{\rrr}$ with heart $\hhh_{\rrr}$. 
\item We have 
$\X_{\rrr}=\bigcup_{n\ge 0}\hhh_{\rrr}[n]\ast \hhh_{\rrr}[n-1]\ast \cdots \ast\hhh_{\rrr}$ and 
$\cal Y_{\rrr}=\bigcup_{n\ge 1}\hhh_{\rrr}[-1]\ast \cdots \ast\hhh_{\rrr}[-n+1]\ast\hhh_{\rrr}[-n]$.
\end{enumerate}
\end{Prop}

Let $\cal S$ be a SMC in $\T$ and let $\cal H=\Filt(\cal S)$. 
We write $\T^{\le n}=\Filt(\cal S[\ge \hspace{-3pt}n])$ and $\T^{\ge n}=\Filt(\cal S[\le \hs n])$.  
The following result is directly from Proposition \ref{Prop:SMCtotstr}.

\begin{Lem}\label{Lem:SMCtotstr}
Let $\T$ be a triangulated category. Let  $\cal S$ be a  SMC of $\T$ and $\hhh=\Filt(\sss)$. Then
\begin{enumerate}[\rm(1)]
 \item We have a  bounded $t$-structure $\T=\T^{\le 0}\perp \T^{\ge 1}$ with heart $\cal H$;
 \item For any $X, Y\in \cal T$, we get $\Hom_{\T}(X[\gg \0], Y)=0$;
 \item For any $Y\in \cal T$, we get $\Hom_{\T}(\cal H[\gg \0], Y)=0$ and $\Hom_{\T}(Y[\gg\0],\hhh)=0$.
 \end{enumerate}
\end{Lem}

Next we recall the notion of simple-minded systems, which is introduced in \cite{KL} and generalized in \cite{CS1}.

\begin{Def}\cite[Definition 2.1]{CS1}
Let $d\ge 0$. We call $\cal S$ a $d$-Simple-minded system (or $d$-SMS) if for any $X,Y\in \cal S$, the following conditions hold.
\begin{enumerate}[\rm(1)]
 \item $\dim \Hom_{\T}(X, Y)=\delta_{X,Y}$;
 \item  If $d\ge1$, then $\Hom_{\T}(X[i], Y)=0$ for any $1\le i\le d$;
 \item $\T=\add \Filt(\{ \cal S[d],\cal S[d-1], \cdots, \cal S\})$.
\end{enumerate} 
\end{Def}
By \cite[Lemma 2.8]{CSP}, the condition (3) above is equivalent to say that $\T=\cal H[d]\ast \cal H[d-1]\ast\cdots\ast\cal H$.

\subsection{Non-positive dg algebras}

Let $A$ be a dg $k$-algebra, that is a graded $k$-algebra with a compatible structure of a complex. We denote by $\D(A)$ the derived category of dg $A$-modules (see \cite{Keller}) and $\D^{\bb}(A)$ the subcategory of $\D(A)$, consisting  of the dg $A$-modules whose total cohomology are finite-dimensional. Let $\per A$ be the perfect derived category, that is, the thick subcategory of $\D(A)$ generated by $A$.

We call a dg $k$-algebra $A$ \emph{proper} if $A\in \D^{\bb}(A)$.
We  say  $A$  is \emph{non-positive}, if $A^{i}=0$ for any $i>0$. In this case, there is a natural map $A \ra \h^{0}(A)$, which is also a morphism of dg algebras. So we may regard any $\h^{0}(A)$-module as a dg $A$-module. 

Let $A$ be a non-positive dg $k$-algebra and $M$ be a dg $A$-module. We define the \emph{standard truncation} $\tau^{\le i}$ and $\tau^{>i}$  by

\[  (\tau^{\le i}M)^{j}:= \begin{cases} M^{j} & \text{for } j<i, \\
\ker d_{M}^{i} & \text{for } j=i, \\
0 & \text{for } j>i. 
\end{cases}      \  \   \  \   \   \   \  \  \   \
(\tau^{>i}M)^{j}:=\begin{cases}
0 & \text{for } j<i, \\
M^{i}/\ker d_{M}^{i} & \text{for } j=i, \\
M^{j} & \text{for } j>i.
\end{cases}\]
Since $A$ is non-positive, then  $\tau^{\le i} M$ and $\tau^{> i}M $ are also dg $A$-modules. Moreover, we have a triangle 
$$\tau^{\le i}M \ra M \ra \tau^{i}M \ra \tau^{\le i}M[1]$$
 in $\D(A)$.
Denote by $\cal S_{A}$ the set of simple $\h^{0}(A)$-modules and we may also regard $S_{A}$ as the set of simple dg $A$-modules (concentrated in degree $0$).  The following results are well-known.

\begin{Lem}\label{Lem:predg}
Let $A$ be a non-positive dg $k$-algebra. Assume $\h^{i}(A)$ is finite-dimensional for any $i\in \Z$. Then 
 \begin{enumerate}[\rm (1)]
  \item $\D^{\bb}(A)$ is Hom-finite. 
  \item $\D^{\bb}(A)=\thick (S_{A})$ and $S_{A}$ is a SMC of $\D^{\bb}(A)$.
 \end{enumerate}
\end{Lem}

\begin{proof}
(1)  is a corollary of  \cite[Theorem 3.1]{Keller}, see also \cite[Proposition 6.12]{AMY}. 

(2) is directly from \cite[Proposition 2.1]{KY2}.
\end{proof}

We end this section by a useful observation that any SMC of $\D^{\bb}(A)$ can be regarded as simple dg $B$-modules for some non-positive dg algebra $B$.

\begin{Prop}\label{Prop:SMCsimple}
Let $A$ be a non-positive proper dg $k$-algebra and let $\cal S$  be a SMC of $\D^{\bb}(A)$.  Then there exists a non-positive dg $k$-algebra $B$ and a triangle equivalence $F: \D^{\bb}(B) \xra{\simeq} \D^{\bb}(A)$ such that $F(\cal S_{B})=\cal S$, where $S_{B}$ is the set of simple dg $B$-modules.
\end{Prop}

\begin{proof}
 There is a bijection
\[ \{\text{SMCs of } \D^{\bb}(A)\} \longleftrightarrow \{\text{silting objects of }\per A \}     
\]
by \cite[Theorem 6.1]{KY} (see also \cite[Theorem 1.2]{SY}). Then there is a silting object $P\in \per A$ corresponding  to $S$. Considering the dg algebra $B':=\shEnd_{A}(P)$, then we have $\h^{i}(B')=0$ for $i>0$ and the  truncation $B:=\tau^{\le 0}B'$ also has a structure of dg $k$-algebra, which is quasi-isomorphic to $B'$. 
Notice that the functor $\RshHom_{A}(P, ?): \D(A)\ra \D(B')$ is a triangle equivalence by \cite[Lemma 4.2]{Keller},  
so there a triangle equivalence $F: \D(B)\ra \D(A)$ which restricts to $\per$ and $\D^{\bb}$. Moreover, by \cite[Theorem 1.1]{SY}, we have that $F(\cal S_{B})=\cal S$. 
\end{proof}

\section{SMC reductions of triangulated categories} \label{Section:SMCreduction}

The aim of this section is to introduce the SMC reduction. It is an operation to construct a new triangulated category  form the given triangulated category  and one of its pre-Simple-minded collections (pre-SMCs). 
One important property is that, under mild conditions, there is a bijection between the SMCs of the  new category  and the SMCs of the original one containing the given pre-SMC.

\subsection{SMC reductions}\label{Section:reduction}
Let $\T$ be a Krull-Schmidt triangulated category and $\cal R$ be a pre-SMC of $\T$ (see Definition \ref{Def:SMC}). We denote by $\SMC \, \T$ the set of SMCs  of $\T$ and by $\SMC_{\rrr}\T$ the set of SMCs of $\T$ containing $\rrr$. 
We define the \emph{SMC reduction}  of  $\T$ with respect to  $\cal R$  as the Verdier quotient $$\cal U:=\T/\thick(\cal R).$$

 By Proposition \ref{Prop:SMCtotstr}, $\thick(\cal R)$ admits a natural $t$-structure $\thick(\cal R)=\X_{\rrr}\perp\Y_{\rrr}$, where $\X_{\rrr}=\Filt(\cal R[\ge \0])$ and $\cal Y_{\rrr}=\Filt(\cal R[<\0])$, whose heart is  denote by $$\cal H_{\cal R}=\Filt(\cal R).$$ 
   Consider the following mild conditions.
\begin{itemize}
 \item[(R1)] $\hhh_{\rrr}$ is contravariantly  finite in $\cal R[>\0]^{\perp}$ and  convariantly finite in ${}^{\perp}\cal R[<\0]$;
 \item[(R2)] For any $X\in \T$, we have $\Hom_{\T}(X, \hhh_{\rrr}[i])=0=\Hom_{\T}(\hhh_{\rrr}, X[i])$ for $i\ll0$.
\end{itemize}

 Notice that by Lemma \ref{Lem:SMCtotstr}, (R2) holds  if there is a SMC of $\T$ containing $\cal R$ .
Let 
$$\cal Z:=\cal R[\ge\0]^{\perp}\cap {}^{\perp}\cal R[\le\0].$$
 Similar to  silting reduction (see \cite[Theorems 3.1 and 3.7]{IY2}), we have the following results.
 \begin{Thm}\label{Thm:SMCbij}
 Assume the assumptions (R1) and (R2) hold. Then 
 \begin{enumerate}[\rm (1)]
   \item  The composition $\cal Z \hookrightarrow \T \ra \cal U$ is an additive equivalence $\cal Z \xra{\simeq} \cal U$;
      \item There is a bijection  
      \[ \SMC_{\rrr}\T \longleftrightarrow \SMC \,\cal U, \]
      sending $\sss\in \SMC_{\rrr}\T$ to $\sss\backslash \rrr \in \SMC\,\cal U$.
 \end{enumerate}
 \end{Thm}
The rest of this section is devoted to the proof of Theorem \ref{Thm:SMCbij}.
We start with the following observation, which is the `dual' of \cite[Proposition 3.2]{IY2}.  
\begin{Prop} \label{Prop:R1}
The following are equivalent.
\begin{enumerate}[\rm(1)]
\item
 $\T=\cal X_{\cal R}\perp \cal X_{\cal R}^{\perp}={}^{\perp}\cal Y_{\cal R}\perp \cal Y_{\cal R}$ are two $t$-structures; 
 \item $\hhh_{\rrr}$ satisfies the conditions (R1) and (R2).
\end{enumerate}
In this case, the heart of $t$-structures in (1) are $\hhh_{\rrr}$.
\end{Prop}
\begin{proof}
We first claim that $\cal X_{\cal R}\cap \X_{\rrr}^{\perp}[1]=\hhh_{\rrr}={}^{\perp}\Y_{\rrr}\cap \Y_{\rrr}[1]$. We only show the first equality since the second one is dual. Since $\rrr$ is a pre-SMC, then $\Hom_{\T}(\X_{\rrr}[1], \hhh_{\rrr})=0$ and  thus  $\hhh_{\rrr}\subset \cal X_{\cal R}\cap \X_{\rrr}^{\perp}[1]$. Now assume $X\in \cal X_{\cal R}\cap \X_{\rrr}^{\perp}[1]$. 
 Since we know $X\in \cal X_{\rrr}=\X_{\rrr}[1]\ast \hhh_{\rrr}$ by Proposition \ref{Prop:SMCtotstr} and $\Hom_{\T}(\X_{\rrr}[1], X)=0$, then it is clear that $X\in \add\hhh_{\rrr}=\hhh_{\rrr}$ (Since $\hhh_{\rrr}$ is the heart of a $t$-structure, so $\add\hhh_{\rrr}=\hhh_{\rrr}$).
 Therefore $\cal X_{\cal R}\cap \X_{\rrr}^{\perp}[1]=\hhh_{\rrr}$.

 $(1)\Rightarrow (2)$  We show (R1).
 For any $X\in \cal R[>\0]^{\perp}=\cal X_{\cal R}^{\perp}[1]$, there is a triangle $$\cal Z[-1]\ra Y\xra{f} X \ra Z $$ with $Y\in \cal X_{\cal R}$ and $Z\in \cal X_{\cal R}^{\perp}$. We claim $Y\in \hhh_{\rrr}$ and $f$ is a right $\hhh_{\rrr}$-approximation of $X$.
 Since $\cal X_{\cal R}^{\perp}[-1]\subset \cal X_{\cal R}^{\perp}[1]$, then $Y\in Z [-1]\ast X \in \cal X_{\cal R}^{\perp}[1]$ and thus $Y\in \cal X_{\cal R}^{\perp}[1]\cap \cal X_{\cal R}=\hhh_{\rrr}$. 
 Since $\Hom_{\T}(\hhh_{\rrr}, Z)=0$, then it follows that $f$ is a right $\hhh_{\rrr}$-approximation. So  
 $\hhh_{\rrr}$ is contravariantly  finite in $\cal R[>\0]^{\perp}$. Dually, $\hhh_{\rrr}$ is  convariantly finite in ${}^{\perp}\cal R[<\0]$. 
 
 We show (R2).
 For any $T\in \T$, consider the triangle  $T'\ra T \ra T'' \ra T'[1]$ with $T'\in \X_{\rrr}$ and $T''\in \X_{\rrr}^{\perp}$. Since $\Hom_{\T}(\hhh_{\rrr}[\gg \0], T')=0$ by Lemma \ref{Lem:SMCtotstr} and $\Hom_{\T}(\hhh_{\rrr}[\ge \0], T'')=0$, then  we know $\Hom_{\T}(\hhh_{\rrr}[\gg\0], T)=0$. The dual argument shows $\Hom_{\T}(T[\gg\0], \hhh_{\rrr})=0$. 

 $(2) \Rightarrow (1)$ We only show $\T=\cal X_{\cal R}\perp \cal X_{\cal R}^{\perp}$ is a $t$-structure, because  the other assertion can be shown  similarly.
 Since $\X_{\rrr}[1]\subset\X_{\rrr}$, it is enough to show $\T=\X_{\rrr}\ast\X_{\rrr}^{\perp}$.
  Let $X\in \T$.  We have $\Hom_{\T}(\hhh_{\rrr}[\ge \hs l], X)=0$ for some $l\in\Z$  by (R2). 
Notice that by   Proposition \ref{Prop:SMCtotstr}, $\X_{\rrr}=\bigcup_{n\ge 0}\hhh_{\rrr}[n]\ast\cdots\ast\hhh_{\rrr}$.
  If $l\le 0$, then we get $X\in \hhh_{\rrr}[\ge\hs l]^{\perp}\subset \hhh_{\rrr}[\ge \0]^{\perp}=\X_{\rrr}^{\perp}$, and thus $X\in \X_{\rrr}\ast \X_{\rrr}^{\perp}$. 
  
Next we  use the induction on $l$ to prove  $X\in \X_{\rrr}\ast \X_{\rrr}^{\perp}$ generally.
 We assume $\hhh_{\rrr}[\ge \hs l-1]^{\perp}\subset \X_{\rrr}\ast\X_{\rrr}^{\perp}$ for some $l>0$.
By assumption (R1),  there exists a triangle $H[l-1]\xra{f} X \ra X' \ra H[l]$ such that  $f$ is a minimal right $(\hhh_{\rrr}[l-1])$-approximation of $X$.
Since $X\in\hhh_{\rrr}[\ge \hs l]^{\perp}$, then $f$ is also a minimal right $(\hhh_{\rrr}[\ge\hs l-1])$-approximation and $\Hom_{\T}(\hhh_{\rrr}[\ge \hs l-1], X')=0$  by Lemma \ref{Lem:notation}. 
By our assumption, $X'\in \X_{\rrr}\ast\X_{\rrr}^{\perp}$. Thus $X\in H[l-1]\ast X'\subset \hhh_{\rrr}[l-1]\ast\X_{\rrr}\ast\X_{\rrr}^{\perp}=\X_{\rrr}\ast\X_{\rrr}^{\perp}$ holds since $\X_{\rrr}$ is extension-closed. 
\end{proof}

 The following proposition shows the first statement of Theorem \ref{Thm:SMCbij}.

\begin{Prop}\label{Prop:decomp.}
 \begin{enumerate}[\rm (1)]
 \item The natural functor $\cal Z \hookrightarrow \T \ra \cal U$ gives an equivalence $\cal Z \xra{\simeq} \cal U$;
    \item We have $\T=\X_{\cal R}\perp\cal Z\perp \Y_{\cal R}[1]$.
\end{enumerate}
\end{Prop}
\begin{proof}
(1) By Propositions \ref{Prop:SMCtotstr} and \ref{Prop:R1},
we have $t$-structures $\thick(\rrr)=\X_{\rrr}\perp\Y_{\rrr}$ and 
$\T=\cal X_{\cal R}\perp \cal X_{\cal R}^{\perp}={}^{\perp}\cal Y_{\cal R}\perp \cal Y_{\cal R}$. Notice that $\X[1]\cap \Y={0}$, then the assertion holds by  Proposition \ref{Thm:IY}.

(2)  It suffices to show $\X_{\rrr}^{\perp}= \cal Z \perp \Y_{\rrr}[1]$. For any $M\in \X_{\rrr}^{\perp}$, there is a triangle $M''[-1]\ra M'\ra M\ra M''$ with $M'\in{}^{\perp}\Y_{\rrr}[1]$ and $M''\in \Y_{\rrr}[1]$ by Proposition \ref{Prop:R1}.
Applying $\Hom_{\T}(\X_{\rrr}, ?)$ to this triangle, it is easy to see $M'\in \X_{\rrr}^{\perp}$. Then $M'\in\X_{\rrr}^{\perp}\cap{}^{\perp}\Y_{\rrr}[1]=\zzz$.
So  $\X_{\rrr}^{\perp}= \cal Z \perp \Y_{\rrr}[1]$.
\end{proof}

In the next part, we study the triangulated structure of $\zzz$, which will be used later.
Since $\cal U$ has a natural structure of triangulated category, then by using the additive equivalence $\cal Z \xra{\simeq} \cal U$, we may also regard $\cal Z$ as a triangulated category. Now we describe the shift functor $\lan1\ran$ in $\cal Z$. 

We define $\langle1\rangle$ on objects of $\zzz$ first. 
For any $X\in \cal Z$, we have $X[1]\in \rrr[>\0]^{\perp}$ and by (R1), there exists a $\hhh_{\rrr}$-approximation of $X[1]$. Define  
  $X\lan 1\ran$ as the third term of the following triangle.
\begin{eqnarray}\label{Omega}
 R_{X} \xra{f_{X}} X[1] \ra  X\lan1\ran \ra R_{X}[1]
 \end{eqnarray}
where $f_{X}$ is the minimal right 
$\hhh_{\rrr}$-approximation of $X[1]$. 
Notice that $ X\lan1\ran$ is defined uniquely up to isomorphism.
 Similarly, we can define $X\lan-1\ran$. Immediately, we have the following observation. 
 
 \begin{Lem}\label{Lem:Omega}
 Let $\lan1\ran$ be defined as above. Then
 \begin{enumerate}[\rm(1)]
 \item For any $X\in \zzz$, we have $X\lan1\ran\in Z$;
 \item For $X\in \Z$ and $n\ge 1$, we have $ X\lan n\ran \in X[n]\ast\hhh_{\rrr}[n]\ast \cdots \ast \hhh_{\rrr}[1]$.
 \end{enumerate}
 \end{Lem}
 
 \begin{proof}
 (1)  Since $X\in \zzz$, then $\Hom_{\T}(\rrr[>\0], X[1])=0$. Notice that $\X_{\rrr}[1]=\Filt(\rrr[>\0])$ and $\X_{\rrr}=\X_{\rrr}[1]\ast\hhh_{\rrr}$ by 
 Proposition \ref{Prop:SMCtotstr},
 then $f_{X}$ in triangle \eqref{Omega} is also a minimal right $\X_{\rrr}$-approximation of $X[1]$ and  $X\lan1\ran\in \X_{\rrr}^{\perp}$    by  Lemma \ref{Lem:notation}  (3).  
 
 On the other hand, since $X\in {}^{\perp}\Y_{\rrr}[1]$ and $R_{X}\in \hhh_{\rrr}\subset {}^{\perp}\Y_{\rrr}$, then $X[1]\in {}^{\perp}\Y_{\rrr}[2]\subset{}^{\perp}\Y_{\rrr}[1]$ and $R_{X}[1]\in {}^{\perp}\Y_{\rrr}[1]$.
 Therefore, $X\lan1\ran\in {}^{\perp}\Y_{\rrr}[1]$ by triangle \eqref{Omega}. So $ X\lan1\ran\in \X_{\rrr}^{\perp}\cap {}^{\perp}\Y_{\rrr}[1]=\zzz$.
 
 (2) For $n\ge 1$, consider the following triangle.
 \begin{eqnarray}\label{map}
 R_{X\lan n-1\ran} \ra X\lan n-1\ran[1] \ra  X\lan n\ran \ra R_{X\lan n-1\ran}[1], \end{eqnarray}
 where $R_{X\lan n-1\ran} \ra X\lan n-1\ran[1] $ is the minimal right $\hhh_{\rrr}$-approximation of $X\lan n-1\ran[1]$, then we have $X\lan n\ran\in X\lan n-1\ran [1]\ast\hhh_{\rrr}[1]$. By induction, it easy to see $X\lan n\ran \in X[n]\ast \hhh_{\rrr}[n]\ast \cdots \ast \hhh_{\rrr}[1]$.
 \end{proof}
 
 Next we define $\lan1\ran$ on morphisms of $\zzz$. Let $s\in \Hom_{\ZZ}(X,Y)$ for any $X, Y$ in $\zzz$.  Consider the following diagram.
 \begin{equation} \label{morphism}
 \begin{aligned}
  \xymatrix{ R_{X} \ar[d]^{h} \ar[r]^{f_{X}}  & X[1] \ar[d]^{s[1]} \ar[r]^{g_{X}}  &  X\lan1\ran \ar[d]^{t}\ar[r] & R_{X}[1] \ar[d]^{h[1]}\\
R_{Y} \ar[r]^{f_{Y}}  & Y[1] \ar[r]^{g_{Y}}  & Y\lan1\ran \ar[r] & R_{Y}[1]
}
\end{aligned}
\end{equation}
 Since $\Hom_{\T}(R_{X}, Y\lan1\ran)=0$, then there exists a morphism $h\in\Hom_{\T}(R_{X}, R_{Y})$, such that $s[1]\circ f_{X}=f_{Y}\circ h$.
 Let $t:  X\lan1\ran\ra Y\lan1\ran$ be a morphism  such that  diagram \eqref{morphism} is commutative. We define $ s\lan1\ran:=t$. The following lemma shows $ s\lan 1\ran$ is well defined.
 
 \begin{Lem}\label{Lem:morphism}
 Let $X,Y\in \zzz$. For any $s\in \Hom_{\zzz}(X, Y)$, 
 $s\lan1\ran$ defined above is determined by $s$ uniquely. 
 \end{Lem}
 
 \begin{proof}
 We first claim the morphism $h$ in diagram \eqref{morphism} is uniquely determined by $s$.
 If there exists $h'\in \Hom_{\T}(R_{X},R_{Y})$ such that $s[1]\circ f_{X}=f_{Y}\circ h'$, then $f_{Y}\circ (h-h')=0$ and moreover, $h-h'$ factors through $Y\lan1\ran[-1]$. But $R_{X}\in \hhh_{\rrr}\subset \X_{\rrr}$ and $ Y\lan1\ran \in \X_{\rrr}^{\perp}$, so $\Hom_{\T}(R_{X},  Y\lan1\ran[-1])=0$. Thus $h=h'$.
 
 Next we show $t$ is unique. If there exists $t':  X\lan1\ran\ra Y\lan1\ran$ such that the diagram \eqref{morphism} commutes. Then 
 we have $(t-t')\circ g_{X}=0$, so $t-t'$ factors through $R_{X}[1]$. But $\Hom_{T}(R_{X}[1],  Y\lan1\ran)=0$, then $t=t'$.
 \end{proof}

 By Lemma \ref{Lem:Omega} and Lemma \ref{Lem:morphism}, it is easy to check that $\lan1\ran: \zzz\ra \zzz$
 is a well-defined functor. 
 Notice that the triangle \eqref{Omega} gives an isomorphism $X[1]\cong X\lan1\ran$ in $\cal U$.  
 
 Next we describe the triangles in $\zzz$.
 Let $X, Y\in\zzz$ and $s\in \Hom_{\zzz}(X,Y)$. Then $s$ induces a triangle $X\xra{s}Y\ra Z\ra X[1]$ in $\T$. 
Consider the right $\hhh_{\rrr}$-approximations of $Z$ and $X[1]$. We have the following commutative diagrams.
\begin{equation} \label{triangle}
 \begin{aligned}
\xymatrix{ & & R_{Z} \ar[d] \ar[r] & R_{X} \ar[d] \\
X \ar[r]^{s} & Y \ar[r]^{t} & Z \ar[r] \ar[d]^{u} & X[1] \ar[d]
\\
& & W \ar[r] &  X\lan1\ran }
\end{aligned}
\end{equation}
 In this case, the following result holds.
 \begin{Prop}\label{Prop:triangles}
Consider the triangulated structure of $\cal Z$ induced by $\cal U$. Then 
\begin{enumerate}[\rm(1)]
\item
The suspension functor of  $\cal Z$ is given by $\lan1\ran$;
\item Let $s:X\ra Y$ be a morphism in $\zzz$. Then the triangle in $\zzz$ induced by $s$  is $X\xra{s}Y\xra{ut} W\ra X\lan1\ran$.
\end{enumerate}
\end{Prop}

\begin{proof}
(1) Directly form  Lemma \ref{Lem:Omega} and Lemma \ref{Lem:morphism}.

(2)  Notice we have isomorphism $Z\cong W$ and $X[1]\cong X\lan1\ran$  in $\cal U$. Moreover, $X\xra{s}Y\xra{ut} W\ra X\lan1\ran$ is a triangle in $\cal U$. 
Since we have $W\in \zzz$ (similar to the proof of Lemma \ref{Lem:Omega}).
Then the assertion holds by the equivalence $\cal Z\simeq \cal U$. 
\end{proof}

Now we are ready to prove Theorem \ref{Thm:SMCbij}.

\begin{proof}[Proof of Theorem \ref{Thm:SMCbij}]
(1) is directly from Proposition \ref{Prop:decomp.} (1).

(2)
Let $\sss\in \SMC_{\rrr}\T$. We first show  that $\sss\backslash \rrr\in \SMC\, \cal U$. 
Since $\thick_{\T}(\sss)=\T$, then $\thick_{\cal U}(\sss\backslash \rrr)=\cal U$. Let $X, Y\in \sss\backslash \rrr$.
It is clear from Definition \ref{Def:SMC} that $\sss\backslash \rrr \subset \zzz$. So by (1), we have 
$$\dim\Hom_{\zzz}(X,Y)=\dim\Hom_{\T}(X,Y)=\delta_{X,Y}.$$
Since $ X\lan n\ran \in X[n]\ast\hhh_{\rrr}[n]\ast \cdots \ast \hhh_{\rrr}[1]$ for $n>0$ by Lemma \ref{Lem:Omega} and $\Hom_{\T}(\hhh_{\rrr}[\ge\0], Y)=0$, then by (1) again,  
$$\Hom_{\cal U}(X[n],Y)=\Hom_{\zzz}(X\lan n \ran,Y)=\Hom_{\T}(X\lan n\ran, Y)=\Hom_{\T}(X[n], Y)=0.$$

So $\sss\backslash\rrr\in \SMC\,\cal U$.
Therefore, sending $\sss\in\SMC_{\rrr}\T$ to $\sss\backslash\rrr\in\SMC\,\cal U$ gives us a well-defined map
$\SMC_{\rrr}\T\ra \SMC\,\cal U$, which is clearly injective.

We show the map is also surjective. Let $\sss_{\cal U}$ be a SMC of $\cal U$.
By (1), we may assume $\sss_{\cal U}\subset \zzz$. In this case, $\sss_{\cal U}$ is also a SMC of $\zzz$.
Let $\sss=\sss_{\cal U}\cup \rrr$. We claim $\sss\in \SMC_{\rrr}\T$.
Since $\rrr$ is a pre-SMC and $\zzz=\rrr[\ge\0]^{\perp}\cap {}^{\perp}\rrr[\le \0]$,  it is clear that $\dim\Hom_{\T}(X, Y)=\delta_{X,Y}$ for any $X,Y\in \cal S$, and $\Hom_{\T}(X[>\0], Y)=0$ for $X\in \rrr$, $Y\in \sss$  or $X\in \sss$, $Y\in \rrr$.
Next we show $\Hom_{\T}(X[>\0],Y)=0$ for any $X, Y\in \sss_{\cal U}$.
Notice that   Lemma \ref{Lem:Omega} (2) implies
$X[n]\in \hhh_{\rrr}[n-1]\ast \cdots \ast \hhh_{\rrr}\ast X\lan n \ran$ for $n>0$.
Since $\Hom_{\T}(\hhh_{\rrr}[\ge\0], Y)=0$, then $\Hom_{\T}(X[n], Y)=\Hom_{\zzz}(X\lan n\ran, Y)=0$ for $n>0$.

To show $\sss$ is a SMC of $\T$, we  are left to show $\T=\thick_{\T}(\sss)$. 
Since $X\lan m \ran\in \thick_{\T}(S)$ for any $X\in \sss_{\cal U}$ and $\thick_{\zzz}(\sss_{\cal U})=\zzz$, then 
$\zzz\subset \thick_{\T}(\sss)$. So 
$\thick_{\T}(\cal Z \cup \rrr)\subset \thick_{\T}(\sss)\subset\T$.
But by Proposition \ref{Prop:decomp.} (2), we have $\thick_{\T}(\zzz\cup\rrr)=\T$, so $\thick_{\T}(\sss)=\T$ and thus $\sss\in \SMC_{\rrr}\T$.  Then the  map $\SMC_{\rrr}\T\ra \SMC\,\cal U$ is bijective.
 \end{proof}

\subsection{Examples}
In this subsection, we consider the application of Theorem \ref{Thm:SMCbij} to non-positive dg algebras. 
We first give the following result.

\begin{Prop}
Let $A$ be a non-positive proper dg $k$-algebra. Let $\cal S$ be a SMC of $\D^{\bb}(A)$ and $\cal R$ be a subset of $\cal S$. Then 
 $\hhh_{\rrr}=\Filt(\rrr)$ satisfies the conditions (R1) and (R2) in Section \ref{Section:reduction}.  
\end{Prop}

\begin{proof}
We know (R2) is true by Lemma \ref{Lem:SMCtotstr}. So we only need to  show (R1). In fact, we show that $\hhh_{\rrr}$ is functorially finite in $\D^{\bb}(A)$. 
By Proposition \ref{Prop:SMCsimple}, we may assume $\sss=\sss_{A}$ is the set of simple dg $A$-modules. In this case, $\hhh=\Filt(\sss)$ is equivalent to $\mod \h^{0}(A)$ (see Lemma \ref{Lem:predg}).

We first claim that $\hhh$ is functorially finite in $\D^{\bb}(A)$.
 Let $M\in \D^{\bb}(A)$. Considering the $P$-resolution ${}_{P}M$ of $M$, then ${}_{P}M\cong M$ in $\D^{\bb}(A)$ and for any $N\in \D^{\bb}(A)$, we have $\Hom_{\D^{\bb}(A)}(M, N)=\Hom_{\mathscr H (A)}({}_{P}M, N)$, where $\mathscr H(A)$ is the homotopy category (see \cite[Section 3]{Keller}). 
 We write ${}_{P}M$ as a $k$-complex and consider the following diagram.
 \[\small{\xymatrix{{}_{P}M:  & \cdots \ar[r]  & P^{-1} \ar[r]^{d^{-1}} & P^{0} \ar[r]^{d^{0}} \ar[dr]^{f} \ar[d]^{g}& P^{1} \ar[r] & \cdots  \\
 & & & L & N \ar@{.>}[l]_{h} & 
 } }\]
where $N:=\frac{P^{0}}{\im d^{-1}+K}$ and $K:= P^{0}\cap A(\op_{i\ge 1} P^{i})$. Then  $N\in\hhh$ and the map  $f:{}_{P}M\ra N$ above  is a morphism of dg $A$-modules. For any $L\in \hhh$ and a morphism $g: {}_{P}M\ra L$ of dg $A$-modules, we have $g(K)=0$, so there exists $h: N\ra L $ such that $g=h\circ f$. Then $f$ is a left $\hhh$-approximation of $M$.
Thus $\hhh$ is a covariantly finite subcategory of $\D^{\bb}(A)$. 
Dually, by using $I$-resolutions, we can show that $\hhh$ is contravariantly finite. Therefore the claim is true.

To show $\hhh_{\rrr}$ is functorially finite in $\D^{\bb}(A)$, it is enough to show $\hhh_{\rrr}$ is functorially finite in $\hhh=\mod \h^{0}(A)$. By Lemma \ref{Lem:lambda} below, we know it is true. Then we finish the proof.
\end{proof}

We need the following well-known fact.
\begin{Lem}\label{Lem:lambda}
Let $\Lambda$ be a finite-dimensional $k$-algebra and let $\rrr$ be a subset of simple $\Lambda$-modules. Then $\hhh_{\rrr}=\Filt(\rrr)$ is functorially finite in $\mod \Lambda$.
\end{Lem}

\begin{proof}
There exists an idempotent $e\in \Lambda$ such that $\rrr=\Top (1-e)\Lambda$. It is well-known that we have a standard recollement of abelian categories (see, for example \cite[Example 2.10]{PV})
  \[ \xymatrixcolsep{4pc}\xymatrix{\mod \Lambda/\Lambda e\Lambda \ar[r]^{i_{*}=\rm{inc.}} &\mod \Lambda \ar@/_1.5pc/[l]_{i^{*}=?\ot_{\Lambda}\Lambda/\Lambda e\Lambda} \ar@/^1.5pc/[l]^{i^{!}=\Hom_{\Lambda}(\Lambda/\Lambda e\Lambda, ?)} \ar[r]^{j^{*}=?\ot_{\Lambda}\Lambda e}  &\mod e\Lambda e \ar@/^-1.5pc/[l]_{j_{!}=?\ot_{e\Lambda e}e\Lambda} \ar@/^1.5pc/[l]^{j_{*}=\Hom_{e\Lambda e}(\Lambda e, ?)}
  }. \] 
Then one can show $i_{*}(\mod \Lambda/\Lambda e\Lambda)=\hhh_{\rrr}$ and by \cite[Proposition 2.8]{PV}, for any $M\in \mod \Lambda$, we have two exact sequences
\[ j_{!}j^{*}(M) \ra M \xra{f} i_{*}i^{*}(M)\ra 0,\]
\[ 	0\ra i_{*}i^{!}(M) \xra{g} M \ra j_{*}j^{*}(M).\] 
It is easy to check that $f$ (resp. $g$) is a left (resp. right) $\hhh_{\rrr}$-approxiamtion of $M$. So $\hhh_{\rrr}$ is functorially finite.
\end{proof}

Next we give some  useful observations, which  allow us to realize the SMC reduction of bounded derived categories as  new bounded derived categories.

\begin{Prop}\label{Prop:SMCex}
Let $A$ be a non-positive proper dg $k$-algebra. Let $e$ be an idempotent.
Assume $e\in A^{0}$.
\begin{enumerate}[\rm(1)]
\item Let $\rrr=\Top (1-e)\h^{0}(A)$. Then the SMC reduction $\D^{\bb}(A)/\thick(\rrr)$ is triangle  equivalent to $\D^{\bb}(eAe)$;
\item
 Assume $eA$ is a pre-SMC in $\D^{\bb}(\mod A)$. Then the SMC reduction $\D^{\bb}(\mod A)/\thick(eA)$ is  triangle equivalent to $\D^{\bb}(B)$, where $B$ is the dg $k$-algebra $\shEnd_{\per A/\thick eA}(A)$.
 \end{enumerate}
\end{Prop}

\begin{proof}
We have a natural derived  Schur functor $F=?\otimes_{A}^{\bf L}Ae: \D(A)\ra \D(eAe)$, which restricts to a functor $F^{\bb}=?\otimes_{A}^{\bf L}Ae: \D^{\bb}(A)\ra \D^{\bb}(eAe)$.
 It is well-known that $F$ admits a left adjoint $G=?\ot_{eAe}^{\bf L}eA$, which is fully faithful  (see for example \cite[Lemma 4.2]{Keller}).

(1)
By Proposition \ref{Prop:idem}, the  functor $F^{\bb}$ induces a triangle equivalence $F^{\bb}: \D^{\bb}(A)/\ker F^{\bb}\xra{\simeq} \D^{\bb}(eAe)$.
Since $\ker F^{\bb}=\{ M\in \D^{\bb}(A)\mid Me=0 \text{ in } \D^{\bb}(eAe) \}$, then by standard truncation, we have $\ker F^{\bb}=\thick (\rrr)$.
So the SMC reduction $\D^{\bb}(A)/\thick(\rrr)$ is equivalent to $\D^{\bb}(eAe)$.

(2)
We claim under our assumption,   $G$ also restricts to $\D^{\bb}$.
Since $eA$ is a pre-SMC, then $\End(eA)=eAe$ is a division ring. Thus 
$\D^{\bb}(eAe)=\per eAe$ and  $eA$ has finite projective dimension as left $eAe$-module, so $G$ also restricts to  $G^{\bb}: \D^{\bb}(eAe)\ra \D^{\bb}(A)$.  Then we have  an adjoint pair $(G^{\bb},F^{\bb})$ between $\D^{\bb}$ and moreover, we have a $t$-structure $(G^{\bb}(\D^{\bb}(eAe)), \ker F^{\bb})$ of $\D^{\bb}(A)$.
So there is a triangle equivalent $\D^{\bb}(A)/G^{\bb}(\D^{\bb}(eAe))\xra{\simeq} \ker F^{\bb}$.
Notice that $G^{\bb}(\D^{\bb}(eAe))=G^{\bb}(\per eAe)=\thick_{A}(eA)$ and by \cite[Corollary 2.12]{KY2} (b), 
 we have $\ker F^{\bb}\cong \D^{\bb}(B)$, where $B$ is the dg $k$-algebra $\shEnd_{\per A/\thick eA}(A)$.
Then the SMC reduction $\D^{\bb}(A)/\thick(eA)$ is equivalent to $\D^{\bb}(B)$.
\end{proof}



Let us consider a concrete  example. 
\begin{Ex}
Let $A$ be a finite-dimensional $k$-algebra presented by a quiver 
\xymatrix{1 \ar@<-1pt>[r]_{\alpha} & 2 \ar@<-3pt>[l]_{\beta}} with relations $\alpha\beta=0=\beta\alpha$.
Let $P_{i}$ (resp. $S_{i}$), $i=1,2$, be the  indecomposable projective (simple) module which  corresponds to the vertex $i$.
 It is easy to check $P_{1}$ is a pre-SMC in $\D^{\bb}(\mod A)$. Then by  Proposition \ref{Prop:SMCex} (2), the SMC reduction  $\D^{\rm b}(\mod A)/\thick(P_{1})$ is equivalent to $\D^{\rm b}(B)$, where $B$ is the dg algebra $k[X]/(X^{2})$ with $\deg X=-1$  and zero differential.
Then  by Theorem \ref{Thm:SMCbij}, we have the following bijection,
\[\xymatrix{ \SMC_{P_{1}} \D^{\rm b}(\mod A)   \ar[r] \ar@{=}[d]& \SMC \, \D^{\bb}(B) \ar[l] \ar@{=}[d] \\
\{\cdots, S_{1}[-2], S_{1}[-1], S_{2}[1], S_{2}[2], \cdots\}   \ar[r] &   \{ k[i]\mid \in
\Z \} \ar[l]
}.\]
We mention that in \cite[Example 2.47]{AI}, the silting quiver of $\per A$ is given and by using Koenig-Yang bijection (see \cite[Theorem 6.1]{KY}), one gets the description of $\SMC\, \D^{\bb}(\mod A)$ and thus the description of $\SMC_{P_{1}} \D^{\rm b}(\mod A)$.
\end{Ex}

\section{Singularity category of  SMC quadruple}

\subsection{Main results}
In this subsection, we introduce the singularity category of a SMC quadruple and show some basic properties of this category. We give the definition of a SMC quadruple first.

\begin{Def}\label{Def:relativeSerre}
We say a quadruple $(\T, \T^{\rm p}, \SSS, \sss)$ is  a \emph{SMC quadruple} if the following conditions are satisfied.
\begin{itemize}
\item[(RS0)] $\T$ is a $k$-linear Hom-finite Krull-Schmidt triangulated category and $\cal T^{\p}$ is a thick subcategory of $\T$;
\item[(RS1)] $\SSS: \T\ra \T$ is a triangle equivalence restricting to an equivalence $\SSS: \T^{\p}\ra \T^{\p}$ and satisfying  a bifunctorial isomorphism for any $X\in \T^{\p}$ and $Y\in \T$:
 \[ D\Hom_{\T}(X, Y) \simeq \Hom_{\T}(Y,\SSS X);\]
 \item[(RS2)] $\cal S$ is a SMC in $\T$ and $\T= \cal {}^{\perp}\cal S[\ge \0]\perp {}^{\perp}\cal S[<\0]=\cal S[\ge \0]^{\perp}\perp\cal S[<\0]^{\perp}$ are co-$t$-structures of $\T$ satisfying ${}^{\perp}\cal S[\ge \hspace{-3pt}0]\subset \T^{\rm p}$ and $\cal S[<\0]^{\perp}\subset \cal T^{\p}$;
 \end{itemize}
Moreover, If $\SSS=[-d]$ for some $d\ge 0$, we call $(\T, \T^{\p}, \cal S)$ a \emph{$(-d)$-CY triple}.
\end{Def}

The definition above is inspired from the following example and we will see  (RS2) plays an important role later. 
\begin{Ex}\label{Ex:fd}
Let $A$ be a finite-dimensional Gorenstein $k$-algebra. 
Then one can show  that  the quadruple $(\D^{\bb}(\mod A), \K^{\bb}(\proj A), \nu, \cal S)$ is a SMC quadruple, where $\nu=?\ot_{A}^{\bf L}DA$ is the Nakayama functor and $\cal S$ is the set of simple $A$-modules.
\end{Ex}

For simplicity, we introduce the following notations for $i\in\Z$.
$$\T_{>i}=\T_{\ge i+1}:=\cal {}^{\perp}\cal S[> \hspace{-3pt}-1-i], \ \ \T_{<i}=\T_{\le i-1}:={}^{\perp}\cal \cal S[<\hspace{-3pt} 1-i];$$
$$\T^{>i}=\T^{\ge i+1}:= \Filt (\sss[<\hs -i]), \  \  \T^{<i}=\T^{\le i-1}:=\Filt(\sss[\ge 1-i]). $$
Let $(\T, \T^{\rm p}, \SSS,\cal S)$ be a SMC quadruple.
Then  we have co-$t$-structures $\T=\T_{>i}\perp\T_{\le i}$  and moreover, $\T_{>i}\subset \T^{\p}$ by (RS2).
Also notice that  we have bounded
 $t$-structures $\T=\T^{\le i}\perp \T^{>i}$ 
 and $\T^{\le i}=\T_{\le i}$
by Lemma \ref{Lem:SMCtotstr}. 
Immediately, we have the following useful  observation.

\begin{Lem}\label{RS3}
Let  $(\T, \T^{\rm p}, \SSS, \sss)$ be a SMC quadruple. Then $(\T^{\p})^{\perp}=0$ in $\T$.
\end{Lem}

\begin{proof}
For any $X\in \T$ and $i\in\Z$,
there exists a triangle $X_{>i}\ra X \ra X_{\le i}\ra X_{>i}[1]$, such that  
$X_{>i}\in \T_{>i}\subset \T^{\p}$ and $X_{\le i}\in \T_{\le i}=\T^{\le i}$ by (RS2). 
If $X\in (\T^{\p})^{\perp}$, then $\Hom_{\T}(X_{>i}, X)=0$ and thus 
$X_{\le i}\cong X\op X_{>i}[1]$ in $\T$.
So $X\in \T^{\le i}$ for any $i\in \Z$. 
Since $\T=\T^{\le 0}\perp \T^{>0}$ is a bounded $t$-structure by Lemma 
 \ref{Lem:SMCtotstr}, then $X\in \bigcap_{i\in\Z}\T^{\le i}=0$ .
\end{proof}

Now we introduce a new class of triangulated categories, which is a generalization of Buchweitz and Orlov's construction of singularity categories.

\begin{Def}
For a SMC quadruple $(\T, \T^{\rm p}, \SSS,\cal M)$, we define the \emph{singularity category} as the Verdier quotient 
\[ \cal \T_{\sg}:= \T/\T^{\rm p}.\]
\end{Def}

One important property of $\T_{\sg}$ is that $\T_{\sg}$ can be realized as a subfactor category of $\T$. To make it clear, let us introduce the following subcategories of $\T$.
\[
\cal F = \T_{>0}^{\perp}\cap {}^{\perp}(\T_{\le -1}\cap \T^{\p}),  \  \  
\cal P= \T_{\ge 0}\cap \T_{\le 0}, \  \
 \hhh=\T^{\ge 0}\cap \T^{\le 0}. \]

It is clear $\cal P$ is just the co-heart of  
the co-$t$-structure of $\T=\T_{>0}\perp\T_{\le 0}$ and $\hhh$ is the heart of the $t$-structure $\T=\T^{\le 0}\perp\T^{>0}$. Our main results in this section is as follows.

\begin{Thm}\label{Thm:singularity}
Let $(\T, \T^{\rm p}, \SSS,\cal S)$ be a SMC quadruple.  Then we have
\begin{enumerate}[\rm (1)]
\item $\cal F$ is a Frobenius extriangulated category with $\Proj \cal F=\cal P$  in the sense of \cite{NP};
\item
The composition
\[ \cal F \subset \T \ra \cal \T_{\sg}  \]
induces an equivalence $\pi: \frac{\cal F}{[\cal P]} \xra{\simeq} \cal \T_{\sg}$. Moreover,  $\T_{\sg}$ has a Serre functor $\SSS[-1]$;
\item If $\SSS=[-d]$, then $\cal F=\cal H[d]\ast\cal H[d-1]\ast\cdots\ast\cal H$  and $\pi(\sss)$ is a $d$-SMS in $\T_{\sg}$.
\end{enumerate}
\end{Thm}

\begin{proof}
(1) and (2) We want to apply Proposition \ref{Thm:IY}. 
Let  $\X=\T_{>0}$
and $\Y=\T_{\le 0}\cap \T^{\p}$.
Then it is easy to check  
$ \cal F= \cal X^{\perp}\cap {}^{\perp}\cal Y[1]$ and $ \cal P=\cal X[1]\cap \cal Y$.
We claim that we have co-$t$-structures $\T^{\p}=\X\perp\Y$ and $\T=\X\perp\X^{\perp}={}^{\perp}\Y\perp\Y$. In fact, we know $\X^{\perp}=\T_{\le 0}$  and  $\T=\cal X\perp \cal X^{\perp}$ is a co-$t$-structure by (RS2).
For any $T\in\T^{\p}$, there exists a triangle $T_{>0}\ra T \ra T_{\le 0}\ra T_{>0}[1]$ such that $T_{>0}\in \T_{>0}$ and $T_{\le 0}\in \T_{\le 0}$. Since $T_{>0}\in \T^{\p}$ by (RS2), so $\T_{\le 0}$ is also in $\T^{\p}$.
Then the co-$t$-structure $\T=\X\perp\X^{\perp}$ restricts to a co-$t$-structure $\T^{\p}=\X\perp(\X^{\perp}\cap \T^{\p})=\X\perp\Y$ of $\T^{\p}$.

 Now we show $\T={}^{\perp}\Y\perp\Y$ is also a co-$t$-structure.
Since $\Y={}^{\perp}\sss[<\0]\cap \T^{\p}$, then $\Y\subset \SSS^{-1}\sss[<\0]^{\perp}$ by (RS1).
Notice that $\sss[<\0]^{\perp}\subset\T^{\p}$ by (RS2), then it is easy to see $\SSS^{-1} \sss[<\0]^{\perp}\subset \Y$ by (RS1).
So $\Y=\SSS^{-1}\sss[<\0]^{\perp}$.
By (RS2), there is a co-$t$-structure $\T=\sss[\ge \0]^{\perp}\perp\sss[<\0]^{\perp}$. 
Then $\T={}^{\perp}\Y\perp\Y$ is also a co-$t$-structure with ${}^{\perp}\Y=\SSS^{-1}\sss[\ge\0]^{\perp}$.

By Proposition \ref{Thm:IY} and Remark \ref{Rem:extri}, we know 
$\cal F$ is a Frobenius extriangulated category with $\Proj \cal F=\cal P$ and 
the composition
$ \cal F \subset \T \ra \cal \T_{\sg} $
induces an equivalence $\pi: \frac{\cal F}{[\cal P]} \xra{\simeq} \cal \T_{\sg}$.

We are left to show the existence of Serre functor in $\T_{\sg}$. 
Let $X,Y\in \T$.
There exist $i\in \Z$ such that $Y\in \T^{>i}$ (because $\T=\T^{\le 0}\perp\T^{>0}$ is a bounded $t$-structure by Lemma \ref{Lem:SMCtotstr}).
By (RS2), there is a triangle
\[ X_{>i} \ra X\ra X_{\le i} \ra X_{>i}[1],\]
with $X_{>i}\in \T_{>i}$ and $X_{\le i}\in \T_{\le i}=\T^{\le i}$.
 Since $\Hom_{T}(X_{\le i}, Y)=0$ and $X_{>i}\in \T^{\p}$, then the morphism $X_{>i}\ra X$ is a local $\T^{\p}$-cover of $X$ relative to $Y$ in the sense of \cite[Definition 1.2]{Am}. Then by \cite[Lemma1.1, Theorem 1.3 and Proposition 1.4]{Am}, we know $\SSS[-1]$ is a Serre functor of $\T_{\sg}$.

(3) For the case $\SSS=[-d]$, we have 
${}^{\perp}\Y=\SSS^{-1}\cal S[\ge\0]^{\perp}=\T^{> -d}$. On the other hand, $\X^{\perp}=\T_{\le 0}=\T^{\le 0}$. 
So $\cal F= \cal X^{\perp}\cap {}^{\perp}\cal Y[1]=\T^{\le 0}\cap \T^{\ge -d}=\cal H[d]\ast\cal H[d-1]\ast\cdots\ast\cal H$ by Proposition \ref{Prop:SMCtotstr}. 

Next we show $\pi(\sss)$ is a $d$-SMS in $\T_{\sg}$. Let $X, Y\in \sss$.  We may assume $\pi(X)$ and $\pi(Y)$ are non-zero objects in $\T_{\sg}\cong\frac{\cal F}{[\cal P]}$.
Since $$\dim\Hom_{\T_{\sg}}(\pi(X),\pi(Y))=\dim\Hom_{\frac{\cal F}{[\cal P]}}(X,Y)\le \dim\Hom_{\T}(X,Y),$$
and $\dim\Hom_{\T}(X,Y)=\delta_{X,Y}$, then we have $\dim\Hom_{\T_{\sg}}(\pi(X),\pi(Y))=\delta_{\pi(X),\pi(Y)}$.
If $d\ge 1$, since  
$\Hom_{\T}(X[i],Y)=0$ for any $1\le i\le d$, then 
$\Hom_{\T_{\sg}}(\pi(X)[i],\pi(Y))=\Hom_{\frac{\cal F}{[\cal P]}}(X[i],Y)=0$.
The fact $\cal F= \cal H[d]\ast\cal H[d-1]\ast\cdots\ast\cal H$
implies that $\frac{\cal F}{[\cal P]}= \cal H[d]\ast\cal H[d-1]\ast\cdots\ast\cal H$. Then $\T_{\sg}=\pi(\cal H)[d]\ast\pi(\cal H)[d-1]\ast\cdots\ast\pi(\cal H)$. So $\pi(S)$ is a $d$-SMS of $\T_{\sg}$.
\end{proof}

We apply Theorem \ref{Thm:singularity} to Example \ref{Ex:fd} and then we have the following well-known result.

\begin{Ex}
Let $A$ be a finite-dimensional Gorenstein $k$-algebra. Then $\cal P=\K^{\bb}(\proj A)$ and $\cal F=\CM A$. By theorem \ref{Thm:singularity}, the natural functor $\CM A\subset \D^{\bb}(\mod A)\ra \D^{\bb}_{\sg}(A)$ gives an equivalence $\un{\CM}A\simeq \D^{\bb}_{\sg}(A)$ and moreover, $\D^{\bb}_{\sg}(A)$ has a Serre functor $?\ot_{A}^{\bf L}DA[-1]$.
\end{Ex}

\subsection{Further properties}

In this subsection, we continue to study the  properties of a SMC quadruple. This part is  technical and abstract, but we will see it is  useful. 
Let  $(\T, \T^{\rm p}, \SSS, \sss)$ be a SMC quadruple. Let $\cal P$ be the co-heart of  the co-$t$-structure  $\T=\T_{>0}\perp\T_{\le 0}$. 
It is clear that $\cal P$ is a subcategory of $\T^{\p}$. We mainly study the properties of $\cal P$.
First we point out that  $\cal P$ is silting in $\T^{\p}$.
\begin{Prop}\label{Prop:silting}
\begin{enumerate}[\rm(1)]
\item $\cal P$ is a silting subcategory in $\T^{\p}$;
\item We have a co-$t$-structure $\T^{\p}=\Filt(\cal P[\le \0])\perp\Filt(\cal P[>\0])$. Moreover,  
$\Filt (\cal P[\le\0])=\T_{\ge 0}$ and $\Filt(\cal P[>\0])=\T_{<0}\cap \T^{\p}$.
\end{enumerate}
\end{Prop}

To prove this proposition, we give two lemmas first.

\begin{Lem}\label{Lem:cothick}
For $X\in \T$, if there exist $i\le j\in \Z$ such that $X\in \T_{\ge i}\cap\T_{\le j}$, then $X\in \thick \cal P$.
\end{Lem}

\begin{proof}
We apply the induction on $j-i$. If $j-i=0$, then $\T_{\ge i}\cap\T_{\le j}=\cal P[-i]$, the assertion is clear. Assume it holds for $j-i<n$, $n> 0$. Now consider the case $j-i=n$. There exists a triangle
\[ X_{<j}[-1]\ra X_{\ge j} \ra X \ra X_{<j}  \]
such that $X_{\ge j}\in \T_{\ge j}$ and $X_{<j} \in \T_{<j}$.  
Since $X, X_{<j}[-1]\in \T_{\le j}$, then  
$\Hom_{\T}(\T_{>j}, X_{<j}[-1])=0=\Hom_{\T}(\T_{>j}, X)$. By the triangle above, we have $\Hom_{\T}(\T_{>j}, X_{\ge j})=0$. So $X_{\ge j}\in \T_{\le j}\cap \T_{\ge j}=\cal P[-j]$.
Since $X_{<j}\in \T_{\ge i}\cap\T_{\le j-1}$, by assumption, $X_{<j}\in \thick \cal P$. Then  $X\in \thick \cal P$.
So the statement is true.
\end{proof}

\begin{Lem}\label{Lem:finitemany}
For any $P\in \T^{\p}$, $\Hom_{\T}(P, \sss[n])\not=0$ for only finite many $n\in \Z$.
\end{Lem}

\begin{proof}
We know $\Hom_{\T}(P, X[\ll\0])=0$ by Lemma \ref{Lem:SMCtotstr}. On the other hand, we have $\Hom_{\T}(P, \sss[n])=D\Hom_{\T}(\sss[n], \SSS P)$  by (RS1), which vanishes for big enough $n$. 
So the statement holds.
\end{proof}

Now we are ready to prove Proposition \ref{Prop:silting}.
\begin{proof}[Proof of Proposition \ref{Prop:silting}]
(1)
Since $\cal P$ is the co-heart of a co-$t$-structure,
 then $\Hom_{\T}(\cal P, \cal P[>\0])=0$. To show $\cal P$ is silting in $\T^{\p}$, it suffices to show $\T^{\p}=\thick\cal P$.
For any $P\in \T^{\p}$, there are only finite many $n\in \Z$ such that $\Hom_{\T}(P, \sss[n])\not=0$ by lemma \ref{Lem:finitemany}. Then there exist $i, j\in \Z$ such that $P\in \T_{\ge i}\cap \T_{\le j}$. By Lemma \ref{Lem:cothick}, $P\in \thick\cal P$. So $\cal P$ is a silting object in $\T^{\p}$.

(2) Since $\cal P$ is silting in $\T^{\p}$, then it is known that $\cal P$ gives us a standard co-$t$-structure $\T^{\p}=\Filt(\cal P[\le \0])\perp\Filt(\cal P[>\0])$ (see \cite[Proposition 2.8]{IY2}).
In the proof of Theorem \ref{Thm:singularity}, we showed the co-$t$-structure $\T=\T_{\ge 0}\perp \T_{<0}$  of $\T$ restricts to a co-$t$-structure $\T^{\p}=\T_{\ge 0}\perp (\T_{<0}\cap T^{\p})$ of $\T^{\p}$. 
Since $\Filt(\cal P[\le \0])\subset \T_{\ge 0}$ and $\Filt(\cal P[>\0])\subset\T_{<0}$, it turns out that these two co-$t$-structure coincide with each other. In particular, $\T_{\ge 0}=\Filt (\cal P[\le\0])$.
\end{proof}

Next we study the relation between $\cal P$ and the standard $t$-structure of $\T=\T^{\le 0}\perp\T^{>0}$.
\begin{Prop}\label{Prop:modP}
 \begin{enumerate}[\rm(1)]
  \item We have $\cal P[\ge\0]^{\perp}=\T^{>0}$ and $\cal P[\le\0]^{\perp}=\T^{<0}$ in $\T$;
  \item The functor $\Hom_{\T}(\cal P, ?): \T \ra \mod\cal P$ restricts to an equivalence form the heart $\cal H$ to $\mod \cal P$.
 \end{enumerate}
\end{Prop}

We first show a lemma.

\begin{Lem} \label{Lem:abel}
\begin{enumerate}[\rm(1)]
\item 
$\cal P$ is a contravariantly finite subcategory of $\T$;
\item $\mod \cal P$ is an abelian category.
\end{enumerate}
\end{Lem}
\begin{proof}
(1)
Because $\T=\T^{\le 0}\perp\T^{>0}$ is a $t$-structure and $\Hom_{\T}(\cal P, \T^{>0})=0$, it suffices to show there exists a right $\cal P$-approximation   for any $X\in T^{\le 0}$. 
Let $X\in T^{\le 0}$.
There is  a  triangle 
$X_{\ge 0}\ra X \ra X_{<0}\ra X_{\ge0}[1]$ with $X_{\ge 0}\in \T_{\ge 0}$ and $X_{<0}\in\T_{<0}$. 
Notice that  $X_{\ge 0}\in \T_{\ge 0}=\Filt(\cal P[\le \0])$ by Proposition \ref{Prop:silting}. Then there is a triangle $Y_{>0}\ra X_{\ge 0}\ra Y_{0}$ such that $Y_{>0}\in \Filt(\cal P[<\0])$ and $Y_{0}\in \Filt(\cal P[\ge\0])$. 
It is easy to check $Y_{0}\in\cal P$. We have the following diagram.
\[ \xymatrix{ Y_{>0} \ar[d]^{f'} & &Z &
\\ X_{\ge 0} \ar[r]^{f} \ar[d]^{g'}& X \ar[r]^{g} \ar[ur]^{\beta} & X_{<0} \ar[r] \ar[u]_{h}& X_{\ge 0}[1] 
\\ Y_{0} \ar[ur]_{\alpha} & & & }\]
Since $X\in \T^{\le 0}=\T_{\le 0}$, then $\Hom_{\T}(\cal P[<\0], X)=0$. Then $f\circ f'=0$ and there is $\alpha\in \Hom_{\T}(Y_{0},X)$ such that $f=\alpha\circ g'$.

We claim $\alpha: Y_{0}\ra X$ is a right $\cal P$-approximation of $X$.
Let $Z$ be the third term of the triangle extended by $\alpha$.
Since $\beta \circ f=\beta\circ \alpha \circ g'=0$, then  there exists $h\in \Hom_{\T}(X_{<0}, Z)$ such that $\beta=h\circ g$.
Since $\Hom_{\T}(\cal P, X_{<0})=0$, then $\Hom_{T}(\cal P, \beta)=0$. So $\alpha$ is right $\cal P$-approximation of $X$.

(2) See \cite[Lemma 4.7]{IY2}.
\end{proof}

Now let us prove Proposition \ref{Prop:modP}.

\begin{proof}[Proof of Proposition \ref{Prop:modP}]
(1) We only show $\cal P[\ge\0]^{\perp}=\T^{>0}$, since
$\cal P[\le\0]^{\perp}=\T^{<0}$ is directly  induced from   Proposition \ref{Prop:silting}.
Since $\cal P[\ge \0]\subset \T_{\le 0}\cap \T^{\p}\subset \T_{\le 0}=\T^{\le 0}$,  then   $\T^{>0}\subset \cal P[\ge\0]^{\perp}$. 

We claim $\cal P[\ge\0]^{\perp}\subset\T^{<0}$. Let $X\in \cal P[\ge\0]^{\perp}$. Consider the following triangle
\[X^{>0}[-1] \ra X^{\le 0} \ra X \ra X^{>0} \]
with $X^{\le 0}\in \T^{\le 0}$ and $X^{>0}\in \T^{>0}$. Since $\Hom_{\T}(\cal P[\ge \0], X^{>0}[-1])=0$ and $\Hom_{\T}(\cal P[\ge \0], X)=0$, then by applying $\Hom_{\T}(\cal P[\ge\0], ?)$ to the triangle above, we have $\Hom_{\T}(\cal P[\ge\0], X^{\le 0})=0$.
On the other hand, by the definition of co-heart, we know $\cal P={}^{\perp}\cal S[\not=\0]$, thus  $\Hom_{\T}(\cal P[<\0], X^{\le 0})=0$. So $\Hom_{\T}(\cal P[n], X^{\le 0})=0$ for any $n\in \Z$. Thus $X^{\le 0}=0$ by Proposition \ref{Prop:silting} and Lemma \ref{RS3}.
 So  $X\cong X^{>0}\in\T^{>0}$.   Then $\cal P[\ge\0]^{\perp}=\T^{>0}$ holds.

(2) We have $\cal H=\Filt(\cal S)=\cal P[\not=\0]^{\perp}$ by (1). For any $P\in \cal P$, consider the following triangle.
\begin{eqnarray}\label{0thcohom} P^{<0}\ra P \ra P^{0} \ra P^{<0}[1]\end{eqnarray}
with $P^{<0}\in \T^{<0}$ and $P^{0}\in \T^{\ge 0}$. Since $\Hom_{\T}(P, \cal S[<\0])=0$ and $\Hom_{\T}(P^{<0}[1], \cal S[<\0])=0$, then $\Hom_{\T}(P^{0}, \cal S[<\0])=0$ and $P^{0}\in \cal H$. 
Let $\cal P^{0}=\{ P^{0}\mid P\in \cal P\}\subset \cal H$ be a subcategory of $\cal H$.
It is easy to check that the functor $(-)^{0}:\cal P\ra \cal P^{0}$ is an equivalence. 
Since $\Hom_{\T}(\T^{<0}, \cal H)=0$, then $\Hom_{\T}(P, \cal H)=\Hom_{\T}(P^{0}, \cal H)$ for any $P\in \cal P$. So
we have the following commutative diagram.
\[\xymatrix{   \cal H \ar[d]_{\Hom_{\T}(\cal P^{0}, ?)} \ar[dr]^{\Hom_{\T}(\cal P, ?)}& \\
\mod \cal P^{0} \ar[r]_{(-)^{0}}^{\simeq} & \mod \cal P}\]
To show $\cal H$ is equivalent to $\mod\cal P$, it suffices to show that $\cal P^{0}$ forms a class of projective generators of $\cal H$. 
For any $X\in \cal H$ and $P\in \cal P$, applying $\Hom_{\T}(?, X)$ to the triangle \eqref{0thcohom}, 
we get $\Hom_{\T}(P^{0}, X[1])=0$ by $\Hom_{\T}(P^{<0}, X)=0$ and $\Hom_{\T}(P[-1], X)=0$. So $P^{0}$ is projective in $\cal H$.

For any $X\in \cal H$. Consider the minimal right $\cal P$-approximation of $X$ ($\cal P$ is a contravariantly finite subcategory of $\T$ by Lemma \ref{Lem:abel}).
\[ Y_{\cal P} \ra X_{\cal P} \ra X \ra Y_{\cal P}[1]. \]
Applying $\Hom_{\T}(\cal P, ?)$ to the triangle, we have long exact sequence
\[ \Hom_{\T}(\cal P, X_{\cal P}[i]) \ra \Hom_{\T}(\cal P, X[i]) \ra \Hom_{\T}(\cal P, Y_{\cal P}[i+1]) \ra \Hom_{\T}(\cal P, X_{\cal P}[i+1]).\]
Since $\Hom_{\T}(\cal P, X_{\cal P}[i])=\Hom_{\T}(\cal P, X[i])=0$ for $i>0$, then $\Hom_{\T}(\cal P, Y_{\cal P}[>\hspace{-3pt}1])=0$.
For the case $i=0$, since $\Hom_{\T}(\cal P, X_{\cal P})\ra \Hom_{\T}(\cal P, X)$ is surjective, then $\Hom_{\T}(\cal P, Y_{\cal P}[1])=0$. So $Y_{\cal P}[1]\in \cal P[\ge \0]^{\perp}=\Filt \sss[<\0]$.
Taking 0-th cohomology, we have an exact sequence $(X_{\cal P})^{0}\ra X \ra 0$. 
So $\cal P^{0}$ is a projective generator of $\cal H$.
\end{proof}

The following Proposition is important in the sequel.

\begin{Prop}\label{Prop:inducedzero}
Let  $(\T, \T^{\rm p}, \SSS, \sss)$ be a SMC quadruple. Let $X\in \T_{\le i}$ and $Y\in \T_{\ge i}$ for some $i\in \Z$.  Then for any  $f\in \rad(X, Y)$ and $S\in \sss$, the induced map $\Hom_{\T}(f, S[-i]): \Hom_{\T}(Y, S[-i])\ra \Hom_{\T}(X, S[-i])$ is zero.
\end{Prop}

\begin{proof}
Let $g\in \Hom_{\T}(Y, S[-i])$. We show $g\circ f=0$.
Consider the following diagram,
\[ \xymatrix{  X_{i} \ar[r]^{\alpha} & X \ar[d]^{f} \ar[r]^{\gamma} & X_{< i} \ar[r] & X_{i}[1] \\
Y_{>i} \ar[r] & Y\ar[d]^{g} \ar[r]^{\beta} & Y_{i}\ar[r] \ar@{.>}[ld]^{h}& Y_{>i}[1]  \\
& S[-i] & &}\]
where $X_{i}\in \T_{\ge i}, X_{<i}\in \T_{<i}$ and 
$Y_{>i}\in \T_{>i}, Y_{i}\in \T_{\le i}$. Since $X\in \T_{\le i}$ and $Y\in \T_{\ge i}$, it is easy to check that $X_{i}, Y_{i}\in \T_{\le i}\cap \T_{\ge i}=\cal P[-i]$.
Notice that $\Hom_{\T}(Y_{>i}, S[-i])=0$ for $Y_{>i}\in \T_{>i}={}^{\perp}\sss[\ge \hs -i]$, 
then there exists $h\in \Hom_{\T}(Y_{0}, S[-i])$ such that $g=h\circ\beta$. 
Since $f\in \rad(X, Y)$, then $\beta\circ f\circ\alpha\in \rad(X_{i}, Y_{i})$, and moreover, we have  $g\circ f\circ \alpha= h\circ \beta\circ f\circ\alpha =0$ by the following Lemma \ref{Lem:radical}.
So $g\circ f$ factors through $\gamma$. But $\Hom_{\T}(X_{<i}, S[-i])=0$ by $X_{< i}\in \T_{\le i}={}^{\perp}\sss[\le \hs -i]$, then $g\circ f=0$.  
\end{proof}
 
 The following lemma is a generalization of a well-known result: for a finite-dimensional $k$-algebra $A$, the radical map $f:Q\ra P $ induces a zero map $\Hom_{A}(f, S)=0$, where $P, Q$ are projective $A$-modules and $S$ is simple. 
  
\begin{Lem}\label{Lem:radical}
Let $P, Q \in \cal P$ and $S\in \cal S$. Let $f\in \rad(Q,P)$, then the induced morphism $\Hom_{T}(f, S):\Hom_{\T}(P,S) \ra \Hom_{\T}(Q,S)$ is zero.
\end{Lem}

\begin{proof}
By Proposition \ref{Prop:modP}, the functor $\Hom_{\T}(\cal P, ?): \cal H\ra \mod\cal P$ is an equivalence. Since $\cal S$ is the set of simples of $\cal H$, then $\Hom_{\T}(\cal P, S)$ is  simple in $\mod \cal P$ for any $S\in \sss$. 
Since $f$ is a radical map, then $\Hom_{\T}(Q, f): \Hom_{\T}(Q,Q)\ra \Hom_{\T}(Q, P)$ is a radical map as $\End_{\T}(Q)$-module. Then 
the composition $ \Hom_{\T}(Q,Q)\ra \Hom_{\T}(Q, P)\ra \Hom_{\T}(Q, S)$ is zero. Consider the image of $1_{Q}\in\Hom_{\T}(Q,Q)$ in the composition, we get that   the induced morphism $\Hom_{T}(f, S):\Hom_{\T}(P,S) \ra \Hom_{\T}(Q,S)$ is also zero.
\end{proof}

\subsection{Independence of  SMC quadruple}

The aim of this subsection is to show 
under certain conditions, being a SMC quadruple is independent of the choice of SMC. 
 Let  $(\T, \T^{\rm p}, \SSS, \sss)$ be a SMC quadruple.  Let $\hhh=\Filt(\sss)$.
We show the following result.
\begin{Thm}\label{Thm:S'}
Let $\sss'$ be another SMC of $\T$. Assume that
\begin{enumerate}[\rm (1)]
\item
$\hhh'=\Filt(\sss')$ is functorially finite;
\item There exists $n\in \Z$ such that $\sss'\subset \hhh[n]\ast\hhh[n-1]\ast\cdots\ast\hhh[-n]$.
\end{enumerate}
  Then $(\T, \T^{\rm p}, \SSS, \sss')$
is also a SMC quadruple.
\end{Thm}

\begin{proof}
To show $(\T, \T^{\rm p}, \SSS, \sss')$
is a SMC quadruple,  we only need to check (RS2) in Definition \ref{Def:relativeSerre} holds, that is,  $\T={}^{\perp}\cal S'[\ge \0]\perp {}^{\perp}\cal S'[<\0]=\cal S'[\ge \0]^{\perp}\perp\cal S'[<\0]^{\perp}$ are co-$t$-structures of $\T$, satisfying ${}^{\perp}\cal S'[\ge \hspace{-3pt}0]\subset \T^{\rm p}$ and $\cal S'[<\0]^{\perp}\subset \cal T^{\p}$. 
We may assume, up to shift, that 
\begin{eqnarray}\label{eqn61}
\hhh' &\subset & \hhh[n]\ast\hhh[n-1]\ast\cdots\ast\hhh.
\end{eqnarray}
Then $\Hom_{\T}(\hhh, \hhh'[<\hs -n])=0$ and $\Hom_{\T}(\hhh',\hhh[<\0])=0$. So in this case, we also have 
\begin{eqnarray}\label{eqn62}
\hhh &\subset & \hhh'\ast\hhh'[-1]\ast\cdots\ast\hhh'[-n]. 
\end{eqnarray}

We prove $\T={}^{\perp}\cal S'[\ge \0]\perp {}^{\perp}\cal S'[<\0]$ is a co-$t$-structure. 
 By Proposition \ref{Prop:SMCtotstr}, we have
\begin{eqnarray}\label{eqn63}
{}^{\perp}\sss'[<\0] &= &\bigcup_{i\ge 0}\hhh'[i]\ast\hhh'[i-1]\ast\cdots\ast\hhh'.
\end{eqnarray}
Then \eqref{eqn61}, \eqref{eqn62} and \eqref{eqn63} imply the following equality.
\begin{eqnarray} \label{eqn60}
{}^{\perp}\sss'[<\0] &=& \bigcup_{i\ge n}\hhh[i]\ast\cdots\ast\hhh[n] \ast\hhh'[n-1]\ast\cdots\ast\hhh'.
\end{eqnarray}

Now fix an integer $l\ge 2n$. Let $\X:=\hhh'[l]\ast\hhh'[l-1]\ast\cdots\ast\hhh'$ and $\Y:={}^{\perp}\X$ be two subcategories of $\T$. 
Since $\hhh'$ is convariantly finite, then $\cal X$ is also convariantly finite (see \cite[Theorem 1.4]{Ch0})  and thus $\T=\Y\ast\X$ is a torsion pair by \cite[Proposition 2.3]{IY1}.  
We claim that $\Y\subset {}^{\perp}\sss'[\ge \0]\ast {}^{\perp}\sss'[<\0]$. Then $\T=\Y\ast\X\subset {}^{\perp}\sss'[\ge \0]\ast {}^{\perp}\sss'[<\0]\subset\T$ and therefore, 
$\T= {}^{\perp}\sss'[\ge \0]\ast {}^{\perp}\sss'[<\0]$ is  a co-$t$-structure.

Now we show the claim. For any $Y\in \Y$,  there exists a triangle 
\begin{eqnarray}\label{eqn6}
Y_{<-l}[-1] \xra{f} Y_{\ge -l} \ra Y \ra Y_{< -l}\end{eqnarray}
such that  $Y_{\ge -l}\in \T_{\ge -l}= {}^{\perp}\sss[\ge \hs l+1]$ and $Y_{<-l}\in\T_{<-l}= {}^{\perp}\sss[\le \hs l]$.
 Since $\T_{<-l}= \Filt \sss[>\hs l]\subset{}^{\perp}\sss'[< \0]$ by  \eqref{eqn60}, then to prove the claim,  it suffices to show $Y_{\ge -l}\in {}^{\perp}\sss'[\ge \0]$. With \eqref{eqn61}, we only need to check the following cases.
\begin{enumerate}[\rm (1)]
 \item[(i)] $\Hom_{\T}(Y_{\ge-l},\sss[i])=0$ for $l<i$; 
 \item[(ii)] $\Hom_{\T}(Y_{\ge-l}, \sss[i])=0$ for $n\le i\le l$;
 \item[(iii)] $\Hom_{\T}(Y_{\ge-l}, \sss'[i])=0$ for $0\le i \le n-1.$
 \end{enumerate}
 Notice that  (i) is clear since  $Y_{\ge -l}\in {}^{\perp}\sss[\ge l+1]$. 
 We show (ii). For any $n\le i\le l$, since $\sss[i]\subset\X$ by \eqref{eqn62}, then $\Hom_{\T}(Y, \sss[i])=0$.
 On the other hand, notice that $Y_{<-l}[-1]\in \T_{\le -l}=\T^{\le -l}$ and $Y_{\ge -l}\in \T_{\ge -l}$. Then $\Hom_{\T}(Y_{<-l}[-1], S[i])=0$ for $n\le i
<l$ and by Proposition \ref{Prop:inducedzero}, $\Hom_{\T}(f, S[l])=0$. 
 Then (ii) is true by triangle \eqref{eqn6}.

We show (iii). 
For    $0\le i\le n-1$, \eqref{eqn61} implies  $S'[i]\subset \hhh[2n-1]\ast \cdots \ast \hhh \subset \T^{\ge -l+1}$ (Because $l\ge 2n$ by our assumption).
So in this case, $\Hom_{\T}(Y_{-l}[-1], \sss'[i])=0$.
Since $Y\in\Y$, then $\Hom_{\T}(Y, \sss'[i])=0$ for $0\le i\le l$.  Then by triangle \eqref{eqn6}, (iii) is true.

So our claim above holds and thus $\T={}^{\perp}\cal S'[\ge \0]\perp {}^{\perp}\cal S'[<\0]$ is a co-$t$-structure.
By \eqref{eqn62}, ${}^{\perp}\sss'[\ge \0]\subset {}^{\perp}\sss[\ge \hs n]$. Since $ {}^{\perp}\sss[\ge \hs n]\subset\T^{\p}$, then ${}^{\perp}\sss'[\ge \0]\subset\T^{\p}$.
Similarly, one can show 
$\T=\cal S'[\ge \0]^{\perp}\perp\cal S'[<\0]^{\perp}$ is a co-$t$-structure of $\T$ and $\sss'[\le \0]^{\perp}\subset\T^{\p}$.
So $(\T, \T^{\rm p}, \SSS, \sss')$
is also a SMC quadruple.
\end{proof}

Immediately form Theorem \ref{Thm:S'} above and Theorem \ref{Thm:singularity} (3), we have the following observation.

\begin{Cor}
 Let  $(\T, \T^{\rm p}, \SSS, \sss)$ be a SMC quadruple. Assume there are only finitely many indecomposable objects in $\T$ (up to isomorphism). Then the functor $\T \ra \T_{\sg}$ induces a well-defined map
 \[ \{ \text{SMCs of } \T \} \longrightarrow  \{ d\text{-SMSs of } \T_{\sg} \}. \]
\end{Cor}

\section{Application to  Gorenstein dg algebras}\label{Section:example}

In this section, we consider the  applications of Theorem \ref{Thm:singularity} to Gorenstein dg $k$-algebra. 
Let $A$ be a dg $k$-algebra.
We  use the setting considered in \cite{J}.  Assume $A$ satisfies the following conditions.
 \begin{enumerate}[\rm(1)]
 \item $A$ is \emph{non-positive};
 \item $A$ is \emph{proper};
 \item $A$ is \emph{Gorenstein}, $i.e.$  $\per A$  coincides with the thick subcategory generated by $DA$.  
 \end{enumerate}

Let $\SSS:=?\ot_{A}^{\bf L}DA$ be the Nakayama functor. Let $\sss=\{S_{i}, 1\le i \le n\}$ be the set of simple  $\h^{0}(A)$-modules. We may also regard $\sss$ as the set of simple dg $A$-modules concentrated in degree $0$. In this case, we have the following observation.
\begin{Prop}\label{Prop:apply}
The quadruple $(\D^{\bb}(A), \per A, \SSS, \sss)$ is a SMC quadruple.
\end{Prop}

To show this proposition, we need prepare some lemmas first. 
\begin{Lem}\label{Lem:apply1}
 Let $X\in \D^{\bb}(A)$. Then the following are equivalent.
\begin{enumerate}[\rm (1)]
\item $X\in \per A$;
\item For all $Y\in \D^{\bb}(A)$, the space $\Hom_{\D^{\bb}(A)}(X, Y[i])$ vanishes for almost all $i\in \Z$.
\end{enumerate} 
\end{Lem}
\begin{Rem}This lemma is known for finite dimensional $k$-algebras (see \cite[Lemma 2.4]{AKLY}). Here we generalize it to any  non-positive proper dg $k$-algebras.
\end{Rem}

\begin{proof}
$(1) \Rightarrow (2) $ Since for any $Y\in \D^{\bb}(A)$ and $i\in \Z$, we have $\Hom_{\D^{\bb}(A)}(A, Y[i])=\h^{i}(Y)$, then it is clear (2) holds for $A$. Thus by d\'evissage, (2) holds for any $X\in \per A=\thick(A)$.

$(2)\Rightarrow (1)$ 
Assume $X\in \D^{\bb}(A)$ satisfies (2).
We construct the following triangles inductively.
\begin{eqnarray} \label{induct}
P_{n}[l_{n}]\xra{f_{n}} X_{n}\ra X_{n+1}\ra P_{n}[l_{n}+1], 
\end{eqnarray}
such that $X_{0}=X$, $P_{n}\in\add A$ and $l_{n}=-\sup\{ l\in\Z \mid \h^{l}(X_{n})\not=0\}$. In addition, the induced map $\h^{-l_{n}}(P_{n})\ra \h^{-l_{n}}(X_{n})$ is the projective cover of $\h^{-l_{n}}(X_{n})$. By our construction, it is easy to see that $l_{0}<l_{1}<l_{2}<\cdots$.  We only need to show $X_{n}=0$ for big enough $n$ and then $X\in \per A$.

We claim  
$$\Hom_{\D^{\bb}(A)}(X_{m}, S[l_{m}])=\Hom_{\D^{\bb}(A)}(X, S[l_{m}]),$$ for any $S\in \sss$. Notice  that 
$\Hom_{\D^{\bb}(A)}(P_{i}[t], S[l_{m}])=0$ for any $i$ and $t<l_{m}$.
We consider two cases $l_{m-1}+1<l_{m} $ and $l_{m-1}+1=l_{m}$. For the first case, we know $l_{n}+1<l_{m}$ for all $n<m$, then  we have 
$$\Hom_{\D^{\bb}(A)}(X_{m},S[l_{m}])=\Hom_{\D^{\bb}(A)}(X_{m_{1}}, S[l_{m}])=\cdots= \Hom_{\D^{\bb}(A)}(X, S[l_{m}])$$
by applying $\Hom_{\D^{\bb}(A)}(?, S[l_{m}])$ to triangles $\eqref{induct}$ for $n<m$.
For the second case, we 
consider the following commutative diagram. 
\[
\xymatrix{ \Hom_{\D^{\bb}(A)}(X_{m-1}[1], S[l_{m}]) \ar[r] \ar[d]^{\simeq}&  \Hom_{\D^{\bb}(A)}(P_{m-1}[l_{m}], S[l_{m}]) \ar[d]^{\simeq} \\
\Hom_{A}(\h^{-l_{m-1}}(X_{m-1}), S) \ar[r]^{\simeq}& \Hom_{A}(\h^{-l_{m-1}}(P_{m-1}), S).
}\]
The left and right arrows are bijective (see for example, \cite[Lemma 4.4]{KN}). Since the lower map is isomorphic by our construction of $P_{m-1}$, so is the upper one. Then we have $\Hom_{\D^{\bb}(A)}(X_{m}, S[l_{m}])=\Hom_{\D^{\bb}(A)}(X_{m-1}, S[l_{m}])$ by triangle \eqref{induct} (taking $n=m-1$). Moreover the claim holds by triangle $\eqref{induct}$.

By our assumption, there exists $N>0$, such that for any $n>N$ and $S\in\sss$, we have $\Hom_{\D^{\bb}(A)}(X, S[n])=0$.
Since there exists $m$ such that $l_{m}>N$. Then by the claim above, $\Hom_{\D^{\bb}(A)}(X_{m}, S[l_{m}])=0$ for all $S\in \sss$.  Then 
it is easy to check $$\Hom_{A}(\h^{-l_{m}}(X_{m}), S)=\Hom_{\D^{\bb}(A)}(X_{m}, S[l_{m}])=0.$$
It suggests $X_{m}$ must be zero. 
Thus $X\in P_{0}[l_{0}]\ast P_{1}[l_{1}]\ast\cdots\ast P_{m}[l_{m}]\subset \per A$.
\end{proof}

\begin{Lem}\label{Lem:apply2}
\begin{enumerate}[\rm(1)]
\item
There is a standard co-$t$-structure of $\per A$ given by $\per A= \Filt (A[<\0])\perp \Filt(A[\ge\0])$. Moreover, we have 
\begin{eqnarray*}
\Filt(A[<\0]) &=&\bigcup_{n>0}\Filt(A[-n])\ast\Filt(A[-n+1]\ast \cdots\ast \Filt(A[-1])); \\
\Filt(A[\ge \0]) &=& \bigcup_{n\ge 0}\Filt(A)\ast\cdots\ast\Filt(A[n-1])\ast\Filt(A[n]).
\end{eqnarray*}
\item
$\Filt(A[<\0])$ is a contravariantly finite subcategory of $\D^{\bb}(A)$ and $\Filt(A[\ge \0])$ is a covariantly finite subcategory of $\D^{\bb}(A)$.
\end{enumerate}
\end{Lem}

\begin{proof}
(1) is well-known, see for example \cite[Proposition 2.8]{IY2}.

(2) We only show  $\Filt(A[<\0])$ is contravariantly finite, since the other statement can be show in a dual way.
Notice that $\Filt(A)=\add A$. Then $\Filt(A[n])$ is a functorially finite subcategory of $\D^{\bb}(A)$ for any $n\in \Z$ and thus,  $\Filt(A[-n])\ast\Filt(A[-n+1]) \ast\cdots\ast\Filt(A[-1])$ is contravariantly finite for $n> 0$ by the dual of  \cite[Theorem 1.4]{Ch0}. Let $M\in \D^{\bb}(A)$. There exists $n> 0$, such that $\Hom_{\D^{\bb}(A)}(A[<\hs -n], M)=0$.
Since 
$$\Filt(A[< \0])=\Filt(A[<\hs -n])\ast( \Filt(A[-n])\ast\Filt(A[-n+1]) \ast\cdots\ast\Filt(A[-1]) ),$$
then Lemma \ref{Lem:notation} (2) suggests that there is a right $\Filt(A[<
\0])$-approximation of $M$.   Therefore  
$\Filt(A[<\0])$ is contravariantly finite.
\end{proof}

Now we prove Proposition \ref{Prop:apply}.
\begin{proof}[Proof of Proposition \ref{Prop:apply}] 
We check the conditions (RS0), (RS1) and (RS2) in Definition \ref{Def:relativeSerre} hold.
(RS0) is clear and in our setting, (RS1) is well-known (see for example \cite[Section 10.1]{Keller}). 
 
We  show  (RS2).
We claim  $\D^{\bb}(A)={}^{\perp}\sss[\ge \0]\perp{}^{\perp}\sss[<\0]$ is a co-$t$-structure with ${}^{\perp}\sss[\ge \0]=\Filt (A[<\0])$. In fact, 
we have a co-$t$-structure $\D^{\bb}(A)=\Filt(A[<\0])\perp \Filt(A[<\0])^{\perp}$ by Lemma \ref{Lem:apply2} and \cite[Proposition 2.3]{IY1}.
Since $\Filt(A[<\0])^{\perp}=\{ M\in\D^{\bb}(A)\mid \h^{>0}(M)=0 \}$, then we have 
$$\Filt(A[<\0])^{\perp}=\Filt(\sss[\ge\0])={}^{\perp}\sss[<\0]$$
by Proposition \ref{Prop:SMCtotstr}.
Thus the claim is ture.

Notice that we have another co-$t$-structure $\D^{\bb}(A)={}^{\perp}\Filt(A[\ge \0])\perp \Filt(A[\ge \0])$.
Since by (RS1), we have a triangle equivalence $\SSS:\D^{\bb}(A)\simeq \D^{\bb}(A)$, then $\SSS$ induces a new co-$t$-structure 
$$ \D^{\bb}(A)={}^{\perp}\Filt(\SSS A[\ge \0])\perp \Filt(\SSS A[\ge \0]),$$
and ${}^{\perp}\Filt(\SSS A[\ge \0])=\Filt(\sss[<\0])=\sss[\ge \0]^{\perp}$ by (RS1). Then we have co-$t$-structure $\D^{\bb}(A)=\sss[\ge\0]^{\perp}\perp\sss[<\0]^{\perp}$ with $\sss[<\0]^{\perp}=\Filt(\SSS A[\ge\0])\subset\per A$. So  $(\D^{\bb}(A), \per A, \SSS, \sss)$ is a SMC quadruple.
\end{proof}

Let $\CM A:=A[< \0]^{\perp}\cap {}^{\perp}A[>\0]$ be the category of  Cohen-Macaulay dg $A$-modules . Then we recover some results obtained in \cite{J} by applying
Theorem \ref{Thm:singularity}.

\begin{Cor}\cite[Theorem 2.4 and 6.5]{J}
Let $A$ be a Gorenstein proper non-positive dg $k$-algebra.
\begin{enumerate}[\rm(1)]
\item
 The composition $\CM A\hookrightarrow \D^{\bb}(A)\ra \D^{\bb}(A)/\per A$ induces a triangle equivalence $\un{\CM}A\xra{\simeq}\D^{\bb}(A)/\per A$.  Moreover, $\un{\CM}A$ admits a Serre functor $?\ot_{A}^{\bf L}DA[-1]$;
 \item If $\SSS=[-d]$, then the set of simple dg $A$-modules is a $d$-SMS  in $\un{\CM}A$.
 \end{enumerate}
\end{Cor}

We end this section by an example. 

\begin{Ex}
Let $A$ be the dg $k$-algebra $k[X]/(X^{3})$ with
$\deg X=-2$ and zero differential. Then $\CM A=\{ M\in \D^{\bb}(A)\mid \h^{i}(M)=0 \text{ for } i> 0 \text{ and } i<-4 \}$. Then the AR quiver of $\un{\CM}A$ is given by the following.

 {\tiny
       \begin{center}
         \begin{tikzpicture}[scale=0.6]
         \draw
         node (kl) at (0,0) {$k$}
         node (k2l) at (-2,0) {$k[2]$}
         node (k4l) at (-4,0) {$k[4]$}
         node (A21l) at (-6,0) {$A_{2}[1]$}
         node (A21r) at (2,0) {$A_{2}[1]$}
         node (k4r) at (4,0) {$k[4]$}
         node (k2r) at (6,0) {$k[2]$}
         node (kr) at (8,0) {$k$}
         node (A2l) at (-1,1) {$A_{2}$}
         node (A22l) at (-3,1) {$A_{2}[2]$}
         node (k1l) at (-5,1) {$k[1]$}
         node (k3l) at (-7,1) {$k[3]$}
         node (k3r) at (1,1) {$k[3]$}
         node (k1r) at (3,1) {$k[1]$}
         node (A22r) at (5,1) {$A_{2}[2]$}
         node (A2r) at (7,1) {$A_{2}$}
         node at (-8, 0.5) {$\dots$}
         node at (9, 0.5) {$\dots$}
         [->] (k3l) edge (A21l) (A21l) edge (k1l) (k1l)              
         edge (k4l) (k4l) edge (A22l) (A22l) edge (k2l)
         (k2l) edge (A2l) (A2l) edge (kl) (kl) edge (k3r)
         (k3r) edge (A21r) (A21r) edge (k1r) (k1r)              
         edge (k4r) (k4r) edge (A22r) (A22r) edge (k2r)
         (k2r) edge (A2r) (A2r) edge (kr);
       \draw[dotted] (-7.9, 1.3)--(-0.7,1.3)--(0.6,-0.3)--(-6.6,-0.3)--(-7.9,1.3);  
              \draw[dotted] (0.2, 1.3)--(7.4,1.3)--(8.7,-0.3)--(1.5,-0.3)--(0.2,1.3);   
         \end{tikzpicture}
         \end{center}}
 \noindent where $A_{2}$ is the dg $A$-module $k[X]/(X^{2})$. 
\end{Ex}

\section{SMC reduction Versus SMS reduction}

\subsection{The SMC reduction of a Calabi-Yau triple} \label{Section:SMCofCY}

Let $(\T, \T^{\p},\cal S)$ be a $(-d)$-CY triple for $d\ge 0$.  Let $\rrr$ be subset of  $\sss$ such that $\hhh_{\rrr}=\Filt(\rrr)$ is functorially finite subcategory of $\T$.
Then $\rrr$ is a pre-SMC of $\T$ and the conditions (R1) and (R2) in Section \ref{Section:SMCreduction} hold.
Let
$$ \cal U=\T/\thick(\rrr)$$
be the SMC reduction of $\T$ with respect to $\cal R$.
 By relative Serre property (RS1), we  have $\T^{\p}\cap \thick(\cal R)^{\perp}= \T^{\rm p}\cap {}^{\perp}\thick(\cal R)$, which will be denoted by $\cal U^{\p}$, that is,
 $$ \cal U^{\p}:=\T^{\p}\cap \thick(\cal R)^{\perp}= \T^{\rm p}\cap {}^{\perp}\thick(\cal R).$$
 This category can be regarded as a full subcategory of $\cal U$ (see \cite[Lemma 9.1.5]{Neeman}).
 
Our aim in this subsection  is to show the SMC reduction of a Calabi-Yau triple gives us a new Calabi-Yau triple.

\begin{Thm}\label{Thm:SMCred}
The triple $(\cal U, \cal U^{\p}, \cal S)$ is a $(-d)$-CY triple.
\end{Thm}

To prove the theorem above, we need the description of $\cal U$ obtained in Section \ref{Section:SMCreduction}. Let 
$$ \cal Z:=\cal R[\ge\0]^{\perp}\cap {}^{\perp}\cal R[\le\0].$$
Then by Theorem \ref{Thm:SMCbij}, there is an equivalence $\zzz\cong\cal U$ and the SMC $\sss$ in $\cal U$ corresponds to SMC $\sss':=\sss\backslash\rrr$ in $\zzz$.
The following lemma implies the triple $(\cal U,\cal U',\sss)$ is equivalent to the triple $(\zzz, \T^{\p}\cap\zzz, \sss')$. So to prove Theorem \ref{Thm:SMCred}, it is equivalent to show $(\cal Z, \T^{\p}\cap\cal Z, \sss')$ is a $(-d)$-CY triple.

\begin{Lem}\label{Lem:Uthick}
We have $\cal U^{\p}=\T^{\p}\cap\cal Z$ as subcategories of $\T$.
\end{Lem}
\begin{proof}
Let $X\in \T^{\p}$. Then $X\in \cal Z$ if and only if $\Hom_{\T}(\cal R[\ge \0], X)=0=\Hom_{\T}(X, \cal R[\le \0])$. By the relative Serre duality (RS1), we have $\Hom_{\T}(\cal R[\ge \0], X)=D\Hom_{\T}(X, \cal R[\ge\hspace{-3pt}-d])$. 
Then $X\in \T^{p}\cap\cal Z$ if and only if $X\in\T^{\p}\cap {}^{\perp}\thick(\cal R)$.
\end{proof}

By (RS2), we have co-$t$-structures  $\T={}^{\perp}\sss[\ge \0]\perp{}^{\perp}\sss[< \0]=\cal S[\ge \0]^{\perp}\perp\cal S[<\0]^{\perp}$. Recall we denote by $\T_{>0}={}^{\perp}\sss[\ge \0]$ and $\T_{\le 0}={}^{\perp}\sss[<\0]$.
For $X\in \T$, there is a triangle 
\begin{eqnarray}\label{eqn1}
 X_{\le 0}[-1]\xra{f} X_{>0} \ra X \ra X_{\le 0},
 \end{eqnarray}
with $X_{>0}\in \T_{>0}$ and $X_{\le 0}\in\T_{\le 0}= \T^{\le 0}$. We may assume that $f\in \rad (X_{\le 0}[-1], X_{>0})$. There is  also a triangle
\[ X'_{\ge 0} \ra X \ra X'_{<0} \ra X'_{\ge 0}[1], \]
with $X'_{\ge 0}\in S[>\0]^{\perp}$ and $X'_{< 0}\in S[\le\0]^{\perp}$.
Then we have the following results.
\begin{Lem}\label{Lem:decomp}
Let $X\in \cal Z$. Then 
\begin{enumerate}[\rm(1)]
\item
$X_{>0}\in \T^{\p}\cap\cal Z$ and $X_{\le 0}\in \cal Z$;
\item $X'_{<0}\in \T^{\p}\cap\cal Z$ and $X'_{\ge 0}\in \cal Z$.
\end{enumerate}
\end{Lem}

\begin{proof}
We only prove (1), since the second one can be shown in a similar way.
We first show $X_{>0}\in \T^{\p}\cap \cal Z$. Since $X_{>0}\in \T^{\p}$ by (RS2) and $\T^{\p}\cap\cal Z=\T^{\rm p}\cap {}^{\perp}\thick(\cal R)$ by Lemma \ref{Lem:Uthick}, it suffices to show $X_{>0}\in  {}^{\perp}\thick(\cal R)$.
Since $X_{>0}\in \T_{>0}={}^{\perp}\sss[\ge \0]$, then 
 $\Hom_{\T}(X_{>0}, \cal R[\ge \0])=0$.
 Because $\Hom_{\T}(X, R[<\hspace{-3pt}-1])=0$ and $\Hom_{\T}(X_{\le 0}, R[<\0])=0$, then
  we have $\Hom_{\T}(X_{>0}, R[<\hspace{-3pt}-1])=0$ by triangle \eqref{eqn1}.

We are left to show $\Hom_{\T}(X_{>0}, R[-1])=0$ for any $R\in \cal R$. Since $X_{\le 0}[-1]\in \T_{\le 1}$, $X_{>0}\in \T_{\ge 1}$ and $f\in \rad(X_{\le 0}[1], X_{>0})$, then the induced map $\Hom_{\T}(f, R[-1])$ is zero by Proposition \ref{Prop:inducedzero}. Since $\Hom_{\T}(X, R[-1])=0$, then
$\Hom_{\T}(X_{>0}, R[-1])=0$ by the triangle \eqref{eqn1}. So  $X_{>0}\in  {}^{\perp}\thick(\cal R)$ and  therefore, $X_{>0}\in \T^{\p}\cap\cal Z$.

Since $X_{>0}\in \T^{\p}\cap {}^{\perp}\thick(\rrr)=\T^{\p}\cap \thick(\rrr)^{\perp}$ and $X\in \zzz$, then it is easy to check $X_{\le 0}\in \zzz$ by applying $\Hom_{\T}(\rrr[\ge \0], ?)$ and $\Hom_{\T}(?, \rrr[\le \0])$ to  \eqref{eqn1}.
Thus the assertion is true. 
\end{proof}

Now we are ready to prove Theorem \ref{Thm:SMCred}.
\begin{proof}[Proof of Theorem \ref{Thm:SMCred}]
It is enough to prove $(\zzz,\zzz\cap\T^{\p},\sss')$ is a $(-d)$-CY triple.

By Lemma \ref{Lem:Uthick}, we know $\T^{\p}\cap \cal Z$ is a thick subcategory of $\cal Z$ and moreover, $P\lan1\ran= P[1]$ for any $P\in \T^{\p}\cap Z$.
So the conditions (RS0) and (RS1) in Definition \ref{Def:relativeSerre} hold directly. Next we show there is a co-$t$-structure $\cal Z={}^{\perp}\sss'\lan\ge\0\ran\perp{}^{\perp}\sss'\lan<\0\ran$ and ${}^{\perp}
\sss'\lan\ge\0\ran\subset\T^{\p}\cap\cal Z$.

Let $X\in \cal Z$. Consider the triangle \eqref{eqn1}, we claim  $X_{>0}\in{}^{\perp}\sss'\lan\ge\0\ran$ and $X_{\le 0}\in{}^{\perp}\sss'\lan<\0\ran$.
Notice that for any $S\in\sss'$ and $n\ge 1$, we have
 $$S\lan n \ran\in S[n]\ast\hhh_{\rrr}[n]\ast \cdots \ast \hhh_{\rrr}[1]$$ by Lemma \ref{Lem:Omega}. 
Then $\Hom_{\T}(X_{>0}, \sss[\ge\0])=0$ implies 
$\Hom_{\zzz}(X_{> 0}, \sss'\lan \ge\0\ran)=0$, that is $X_{>0}\in{}^{\perp}\sss'\lan\ge\0\ran$.
 Similarly,   $X_{\le 0}\in {}^{\perp}\sss'\lan<\0\ran$ by the fact that    $S\lan -m\ran=\hhh_{\rrr}[\le\hspace{-3pt} -1]\ast \cdots \ast \hhh_{\rrr}[\le \hs -m]\ast S[-m]$ for $m>0$ and $X_{\le 0}\in \T^{\le 0}$. Thus we have
 $$\cal Z={}^{\perp}\sss'\lan\ge\0\ran\ast{}^{\perp}\sss'\lan<\0\ran.$$
  Notice that $\sss'$ is a SMC in $\zzz$ by Theorem \ref{Thm:SMCbij},  then ${}^{\perp}\sss'\lan<\0\ran=\Filt(\sss'\lan \ge\0\ran)$ and therefore, $\Hom_{\zzz}({}^{\perp}\sss'\lan\ge\0\ran,{}^{\perp}\sss'\lan<\0\ran)=0$. So the claim holds and  $\cal Z={}^{\perp}\sss'\lan\ge\0\ran\perp{}^{\perp}\sss'\lan<\0\ran$ is a co-$t$-structure.

 Assume $X\in {}^{\perp}\sss'\lan\ge\0\ran$, consider the triangle \eqref{eqn1}, since $X_{>0}\in \T^{\p}$ by (RS2) and we have shown $X_{\le 0}\in{}^{\perp}\sss'\lan<\0\ran$ above, then $\Hom_{\zzz}(X, X_{\le 0})=0$ and thus  $X$ is a direct summand of $X_{>0}$.  So $X\in \T^{\p}$ and  ${}^{\perp}
\sss'\lan\ge\0\ran\subset\T^{\p}\cap\cal Z$.
 
 Similarly, one can show $\cal Z=\sss'\lan\ge\0\ran^{\perp}\perp\sss'\lan<\0\ran^{\perp}$ is also a co-$t$-structure with $\sss'\lan<\0\ran^{\perp}\subset\zzz\cap\T^{\p}$. Thus  $(\zzz,\zzz\cap\T^{\p},\sss')$   is a $(-d)$-CY triple and so is $(\cal U, \cal U^{\p}, \cal S)$.
\end{proof}

\subsection{SMC reduction reduces SMS reduction}
In this section, we study the relation between SMC reduction and SMS reduction introduced in \cite{CSP}.
Let $(\T, \T^{\p},\cal S)$ be a $(-d)$-CY triple for $d\ge 0$. Let $\hhh=\Filt(S)$. Let $\rrr$ be a subset of  $\sss$ such that $\hhh_{\rrr}=\Filt(\rrr)$ is functorially finite subcategory of $\T$.

The singularity category
 $\T_{\sg}$ is a $(-d-1)$-CY triangulated category and $\sss$ is a $d$-SMS in $\T_{\sg}$  by Theorem \ref{Thm:singularity}.  Moreover, we may regard $\T_{\sg}$ as a subfactor category of $\T$, that is 
 $$ \frac{\cal F}{[\cal P]}\simeq \T_{\sg},$$
 where $\cal F=\cal H[d]\ast\cal H[d-1]\ast\cdots\ast\cal H$, and $\cal P=\T_{\ge 0}\cap \T_{\le 0}$. By this description, it is easy to check $\hhh_{\rrr}$ is also functorially finite in $\T_{\sg}$.
  Let 
$$(\T_{\sg})_{\cal R}=\{X\in \T_{\sg}\mid \Hom_{\T_{\sg}}(\cal R[i], X)=\Hom_{\T_{\sg}}(X,\cal R[-i])=0, \text{ for } 0\le i\le d \}.$$ 
Then we regard $(\T_{\sg})_{\cal R}$ as the SMS reduction of $\T_{\sg}$ with respect to $\cal R$ in the sense of \cite{CSP}. 
By \cite[Theorems 4.1 and 5.1]{CSP}, $(\T_{\sg})_{\cal R}$ has a structure of triangulated category. 

In Section \ref{Section:SMCofCY}, we have shown  the triple $(\cal U, \cal U^{\p},\sss)$ of the reduction of $(\T,\T^{\p},\sss)$ is still a $(-d)$-CY triple (Theorem \ref{Thm:SMCred}).
Our main result of  this subsection is that  the SMS reduction of the singularity category coincides with the singularity category of the SMC reduction
in the  following sense.

\begin{Thm}\label{Thm:mainresult}
There is a triangle equivalence from $\cal U_{\sg}=\cal U/\cal U^{\p}$ to $(\T_{\sg})_{\cal R}$.
\end{Thm}

Recall  we  may regard the triple  $(\cal U, \cal U^{\p}, \cal S)$ as $(\cal Z, \cal Z\cap\T^{\p}, \sss')$. Let $\cal H'=\Filt_{\zzz} \sss'$.
Then $\cal U_{\sg}\cong \cal Z/(\cal Z\cap\T^{\p})$ is equivalent to $\frac{\cal F_{\zzz}}{[\cal P_{\zzz}]}$ by 
Theorem \ref{Thm:singularity}, where 
$\cal F_{\zzz}= \cal H'\lan d\ran\ast \cal H'\lan d-1\ran\ast\cdots \ast \cal H'$ 
and $\cal P_{\cal Z}={}^{\perp}\cal H'[\not=\0]$.

We first show the functor $\zzz\hookrightarrow\T\ra\T_{\sg}$ induces a well-defined functor $\zzz\ra(\T_{\sg})_{\rrr}$. Before this, we 
give some general results, which will be used later. 

\begin{Lem}\label{Lem:cocone}
\begin{enumerate}[\rm (1)]
\item
Let $X\in \T$ and $Y\in \T^{\le 0}$. Then 
any morphism in $\Hom_{\T_{\sg}}(X, Y)$ has a representative of the form $X\xra{f} Z \xleftarrow{s} Y$ such that the cocone of $s$ belongs to $\T^{\p}\cap \T^{\le 0}$;
\item 
Let $X\in\T^{\ge 0}$ and $Y\in \T$. Then any morphism in $\Hom_{\T_{\sg}}(X, Y)$ has a representative of the form $X \xleftarrow{t} Z \xra{f} Y$ such that the cone of $t$ belongs to $\T^{\p}\cap \T^{\ge 0}$.
\end{enumerate}
\end{Lem}

\begin{proof}
We only show (1), since (2) can be shown in a similar way.
Any morphism $X\ra Y$ in $\T_{\sg}$ can be written as $X \xra{f} Z \xleftarrow{s} Y$, such that there is a triangle 
\[\xymatrix{ W \ar[r]^{g} & Y \ar[r]^{s} & Z \ar[r] &W[1]}  \]
with $W\in \T^{\p}$. Consider the triangle $W_{>0}\ra W\ra W_{\le 0}\ra W_{>0}[1]$ with $W_{>0}\in \T_{>0}$ and $W_{\le 0}\in T_{\le 0}$.
Notice that $W_{\le 0}\in \T^{\p}$  by the fact $\T_{>0}\subset\T^{\p}$ and the triangle above.
 Since $Y\in \T^{\le 0}=\T_{\le 0}$, then $\Hom_{\T}(W_{>0}, Y)=0$ and $g$ factors though $W\ra W_{\le 0}$.
Thus we obtain the following commutative diagram of triangles.
\[ \xymatrix{ W \ar[r]^{g} \ar[d]& Y \ar[r]^{s} \ar@{=}[d] & Z \ar[r] \ar[d]^{h}& W[1]\ar[d] \\
W_{\le 0} \ar[r] & Y \ar[r]^{hs} & Z' \ar[r] & W_{\le 0}[1]
  }\]
  The morphism  $X \xra{f} Z \xleftarrow{s} Y$ is equivalent to $X \xra{hf} Z' \xleftarrow{hs} Y$, and in this case,  the cocone $W_{\le 0}$ of $hs$ belongs to $\T^{\p}\cap \T^{\le 0}$,  so the assertion follows.
\end{proof}

The following observation is useful.
\begin{Prop}\label{Prop:inducedbij}
\begin{enumerate}[\rm(1)]
\item
The functor $\T\ra \T_{\sg}$ induces a bijection (resp. surjection) 
$\Hom_{\T}(X, Y)\ra \Hom_{\T_{\sg}}(X, Y)$ for $X\in \T^{\ge -d+1}$ (resp. $X\in \T^{\ge -d}$) and $Y\in \T^{\le 0}$;
\item The functor $\T\ra \T_{\sg}$ induces a bijection (resp. surjection) 
$\Hom_{\T}(X, Y)\ra \Hom_{\T_{\sg}}(X, Y)$ for
$X\in \T^{\ge 0}$ and $Y\in \T^{\le d-1}$ (resp. $T\in \T^{\le d}$).
\end{enumerate}
\end{Prop}

\begin{proof}
We only prove the first statement and (2) is similar by using Lemma \ref{Lem:cocone} (2).
 We first show $\Hom_{\T}(X, Y) \ra \Hom_{\T_{\sg}}(X,Y)$ is surjective for $X\in \T^{\ge -d}$ and $Y\in \T^{\le 0}$.  By Lemma \ref{Lem:cocone} (1),  any morphism in $\Hom_{\T_{\sg}}(X, Y)$ has a representative  $X\xra{f} Z \xleftarrow{s} Y$ such that the cocone $W$ of $s$ is in $\T^{\p}\cap \T^{\le 0}$,
then we have the following  exact sequence 
\[ \Hom_{\T}(X, Y) \ra \Hom_{\T}(X, Z) \ra \Hom_{\T}(X, W[1]).\]
Since $X\in \T^{\ge -d}$ and $W\in \T^{\p}\cap \T^{\le 0}$, then by relative Serre duality (RS1), we have $\Hom_{\T}(X, W[1])=D\Hom_{\T}(W, X[\le\hspace{-3pt}-d-1])=0$. 
So there exists $g\in \Hom_{\T}(X, Y)$ such that $f=s\circ g$.
Then the morphism $X\xra{f} Z \xleftarrow{s} Y$ is equivalent to $X\xra{g}Y$ in $\T_{\sg}$ and moreover,   $\Hom_{\T}(X, Y) \ra \Hom_{\T_{\sg}}(S,T)$ is surjective.

Next we show $\Hom_{\T}(X, Y) \ra \Hom_{\T_{\sg}}(X,Y)$ is injective if $X\in \T^{\ge -d+1}$. 
Assume $f\in \Hom_{\T}(X,Y)$ is zero in $\T_{\sg}$, then it factors though some $P\in \T^{\p}$.
We may assume $P\in \T^{\p}\cap \T^{\le 0}$ by the proof of Lemma \ref{Lem:cocone} (1). Then by (RS1), $\Hom_{\T}(X, P)=D\Hom_{\T}(P, X[-d])=0$ since $X\in \T^{\ge -d+1}$.
Thus $f$ is zero in $\T$. So the statement follows.
\end{proof}

The following lemma suggests the existence of functor from $\zzz$ to $(\T_{\sg})_{\rrr}$ directly.

\begin{Lem}\label{Lem:Nbij}
Let $X\in \cal Z$, then
\begin{enumerate}[\rm (1)]
\item The map $  \Hom_{\T}(\cal R[i], X) \ra \Hom_{\T_{\sg}}(\cal R[i], X)$
is bijective (resp. surjective) for $i\le d-1$ (resp. $i\le d$). In particular, $\Hom_{\T_{\sg}}(\cal R[i], X)=0$ for $0\le i\le d$;
\item The map $  \Hom_{\T}(X, \cal R[-i]) \ra \Hom_{\T_{\sg}}(X, \cal R[-i])$ is bijective (resp. surjective) for $i\le d-1$ (resp. $i\le d$).  In particular, $\Hom_{\T_{\sg}}(X, \cal R[-i])=0$ for $0\le i\le d$.
\end{enumerate}
\end{Lem}

\begin{proof}
We only shown (1), since (2) is similar by using Lemma \ref{Lem:decomp} (2) and Proposition \ref{Prop:inducedbij} (2). The triangle \eqref{eqn1} induces a commutative diagram as follows,
\[\xymatrix{  \Hom_{\T}(\cal R[i], X) \ar[r] \ar[d]& \Hom_{\T}(\cal R[i], X_{\le 0}) \ar[d]\\
\Hom_{\T_{\sg}}(\cal R[i], X)\ar[r] & \Hom_{\T_{\sg}}(\cal R[i], X_{\le 0}) }\]
The upper map is bijective since $X_{>0}\in \T^{\p}\cap \cal Z \subset \thick(\cal R)^{\perp}$ by Lemma \ref{Lem:decomp} (1) and the lower map is bijective since $X\ra X_{\le 0}$ becomes an isomorphism in $\T_{\sg}$. Since $X_{\le 0}\in \T_{\le 0}=\T^{\le 0}$, then the right map is bijective (resp. surjective) for $i\le d-1$ (resp. $i\le d$) by Proposition \ref{Prop:inducedbij} (1), so is the left one.
Since $X\in \cal Z$, then $\Hom_{\T}(\cal R[\ge \0], X)=0$. So $\Hom_{\T_{\sg}}(\cal R[i], X)=0$ for $0\le i \le d$.
\end{proof}

The following proposition shows we have a triangle functor  $\zzz\ra (\T_{\sg})_{\rrr}$.
\begin{Prop}\label{Prop:inducedwell}
The composition of functors $\cal Z \hookrightarrow \T \xra{\pi} \T_{\sg}$ induces a well-defined triangle functor $\rho: \cal Z\ra (\T_{\sg})_{\cal R}$.
\end{Prop}

\begin{proof}
By Lemma \ref{Lem:Nbij}, it is easy to see $\rho(\cal Z)\subset (\T_{\sg})_{\cal R}$. So $\rho: \cal Z\ra (\T_{\sg})_{\cal R}$ is  well-defined. We show it is a triangle functor.

First we claim $\rho$ commutes with shift functors. Let $X\in \zzz$. Then $X\lan 1\ran$ is defined by the following triangle (see Section \ref{Section:SMCreduction}).
\begin{eqnarray}\label{eqnomega}
 R_{X} \xra{f_{X}} X[1] \ra  X\lan1\ran \ra R_{X}[1],
 \end{eqnarray}
where $R_{X} \xra{f_{X}} X[1]$ is the right $\hhh_{\rrr}$-approximation of $X[1]$ in $\T$.
Now we consider the triangle \eqref{eqnomega} in $\T_{\sg}$.
By the equivalence $\frac{\cal F}{[\cal P]}\simeq \T_{\sg}$ (Theorem \ref{Thm:singularity}), it is clear that
 $R_{X}\xra{f_{X}}X[1]$ is  also a right $\hhh_{\rrr}$-approximation of $X[1]$ in $\T_{\sg}$. Then $\rho(X\lan1\ran)=\rho(X)\lan1\ran$ in $(\T_{\sg})_{\rrr}$ (see \cite[Definition 4.2]{CSP} for the shift functor of $(\T_{\sg})_{\rrr}$).

Next we show $\rho$ sends triangles in $\zzz$ to triangles in $(\T_{\sg})_{\rrr}$.
Let $s:X\ra Y$ be a morphism in $\zzz$. 
Consider the commutative diagrams \eqref{triangle}, then $X\xra{s}Y\ra W\ra X\lan1\ran$ is the triangle induced by $s$ in $\zzz$ by Proposition \ref{Prop:triangles}. In fact, every triangle in $\zzz$ can be obtained in this way. Now we consider the diagrams \eqref{triangle} in $\T_{\sg}$. We have shown that $R_{Z}\ra Z$ and $R_{X}\ra X[1]$ are right $\hhh_{\rrr}$-approximations in $\T_{\sg}$ above.
Then by the construction of triangles of $(\T_{\sg})_{\rrr}$, we know $X\xra{s}Y\ra W\ra X\lan1\ran$ is the triangle given by $s$ in $(\T_{\sg})_{\rrr}$ (see \cite[Theorem 4.1 and Definition 4.4]{CSP}). Then $\rho$ sends triangles to triangles.

So $\rho$ is a triangle functor and the assertion is true.
\end{proof}

Now we are ready to prove our main result. 

\begin{proof}[Proof of Theorem \ref{Thm:mainresult}]
The natural functor $\rho:\zzz\ra(\T_{\sg})_{\rrr}$ is a triangle functor by  Proposition \ref{Prop:inducedwell}.
Since $\rho(\T^{\p})=0$, then $\rho$ induces a triangle functor $\tilde{\rho}:\zzz/(\zzz\cap\T^{\p})\ra (\T_{\sg})_{\rrr}$.
Since $\zzz/(\zzz\cap\T^{\p})$ is equivalent to $\frac{\cal F_{\zzz}}{[\cal P_{\zzz}]}$ by 
Theorem \ref{Thm:singularity}, we have a  functor $\frac{\cal F_{\zzz}}{[\cal P_{\zzz}]}\ra (\T_{\sg})_{\rrr}$, which is also denoted by $\tilde{\rho}$.
We claim $\tilde{\rho}$ is fully faithful and dense.

Let $\sss'=\sss\backslash\rrr$ and $\hhh'=\Filt_{\zzz}(\sss')$.
Then $\sss'$ is a SMC in $\zzz$  by Theorem \ref{Thm:SMCbij} and moreover, $\sss'$ is a $d$-SMS in  $\frac{\cal F_{\zzz}}{[\cal P_{\zzz}]}$ by Theorem \ref{Thm:singularity}. So
 $$\frac{\cal F_{\zzz}}{[\cal P_{\zzz}]}= \cal H'\lan d\ran\ast \cal H'\lan d-1\ran\ast\cdots \ast \cal H'.$$
On the other hand,  $\rho (\sss')$ is a $d$-SMS in $(\T_{\sg})_{\rrr}$ by \cite[Theorem 6.6]{CSP} and thus by \cite[Lemma 2.8]{CSP}, we have
$$ (\T_{\sg})_{\rrr}=\rho(\cal H')\lan d\ran\ast \rho(\cal H')\lan d-1\ran\ast\cdots \ast \rho(\cal H').$$
Then it is clear that $\tilde{\rho}$ is dense. 
We are left to show $\tilde{\rho}$ is fully faithful. Let $X, Y\in \sss'$. We may assume $X, Y\not\in \cal P_{\zzz}$.
 It is enough to show 
\begin{eqnarray}\label{eqneqn}
\Hom_{\frac{\cal F_{\zzz}}{[\cal P_{\zzz}]}}(X\lan i\ran, Y\lan j\ran)=\Hom_{(\T_{\sg})_{\rrr}}(\tilde{\rho}(X)\lan i\ran, \tilde{\rho}(Y)\lan j\ran)
\end{eqnarray}
for any $i, j\in\Z$. Let $t=j-i$. Notice that if $t<0$, then the both sides of equation \eqref{eqneqn} are zero.
If $t=0$. Since $$\dim \Hom_{\frac{\cal F_{\zzz}}{[\cal P_{\zzz}]}}(X, Y)=\dim \Hom_{\zzz}(X,Y)=\delta_{X,Y},$$
and $\Hom_{\zzz}(X,Y)=\Hom_{\T_{\sg}}(\rho(X),\rho(Y))$ by Proposition \ref{Prop:inducedbij}, then \eqref{eqneqn} holds.

If $t>0$. Notice that $Y\lan t\ran \in Y[t]\ast\hhh_{\rrr}[t]\ast \cdots \ast \hhh_{\rrr}[1]$ by Lemma \ref{Lem:Omega}.
Then there is a triangle $Y[t]\ra Y\lan t\ran \ra Z\ra Y[t+1]$  in $\T$ such that $Z\in \hhh_{\rrr}[t]\ast \cdots \ast \hhh_{\rrr}[1]\subset\T^{\le -1}$. Then by Proposition  \ref{Prop:inducedbij} and five lemma, one can show 
 $$\Hom_{\T}(X, Y\lan t\ran)=\Hom_{\T_{\sg}}(X, Y\lan t\ran).$$
Because $\Hom_{\zzz}(\cal P_{\zzz}, Y\lan t\ran)=0$ by our constructionof $\cal P_{\zzz}$,  then 
$\Hom_{\frac{\cal F_{\zzz}}{[\cal P_{\zzz}]}}(X, Y\lan t\ran)=\Hom_{\zzz}(X, Y\lan t\ran)$. So the equation \eqref{eqneqn} is true.  Then $\tilde{\rho}$ is fully faithful.
 
 Thus $\tilde{\rho}:\frac{\cal F_{\zzz}}{[\cal P_{\zzz}]}\ra (\T_{\sg})_{\rrr}$ gives a triangle equivalence and the theorem holds.
\end{proof}

We finish this paper by consider some examples. 
\begin{Ex}
Let $A$ be a finite-dimensional symmetric $k$-algebra and let $e$ be an idempotent. Let $S_{e}=\Top (1-e)A$. Then by Proposition \ref{Prop:SMCex} (1), the SMC reduction of $\D^{\bb}(\mod A)$ with respect to $S_{e}$ is triangle equivalent to $\D^{\bb}(\mod eAe)$. Then by Theorem \ref{Thm:mainresult}, we have the following commutative diagram,

 {\small \[\xymatrixcolsep{5.5pc}\xymatrixrowsep{4pc}\xymatrix{   \D^{\bb}(\mod A) \ar[r]^{\text{sing.   category}}  \ar[d]_{\text{SMC reduction}}& \D_{\sg}(A) \ar@{~>}[d]^{\text{SMS reduction}}\\ \cal \D^{\bb}(\mod eAe) \ar[r]^{\text{sing. category}} & \D_{\sg}(eAe) \cong (\D_{\sg}(A))_{S_{e}}
}\] }
We point out that the left map is given by the functor $?\otimes^{\bf L}_{A}Ae: \D^{\bb} ({\rm mod} A)\ra \D^{\bb}({\rm mod} eAe)$ and, the upper and lower maps are given by the Verdier quotient. But the right map is usually not given by  functors.
\end{Ex}

Next we consider a concrete algebra  $A$ and check the equivalence $ \D_{\sg}(eAe) \cong (\D_{\sg}(A))_{S_{e}}$ by comparing the   AR quivers of them.

\begin{Ex}\label{Ex:end}
Let $A$ be the $k$-algebra given by the quiver 
\xymatrix{1\ar@/^0.6pc/[r]^{\alpha_{1}}  & 2 \ar@/^0.6pc/[r]^{\beta_{1}} \ar@/^0.6pc/[l]^{\alpha_{2}}&3 \ar@/^0.6pc/[l]^{\beta_{2}}},
 with relations $\{\alpha_{1}\alpha_{2}\alpha_{1}, \beta_{2}\beta_{1}\beta_{2}, \alpha_{1}\beta_{1}, \beta_{2}\alpha_{2}, \alpha_{2}\alpha_{1}-\beta_{1}\beta_{2} \}$.
  Let $S_{i}$ be  the simple $A$-modules at vertices $i$ ($i=1,2,3$)  and let  $P_{1}=\begin{smallmatrix} 1 \\ 2\\ 1\end{smallmatrix}$ (resp. $P_{2}=\begin{smallmatrix} & 2 & \\ 1 & &3 \\ &2& \end{smallmatrix}$, $P_{3}=\begin{smallmatrix} 3 \\ 2\\ 3\end{smallmatrix}$)
be the indecomposable projective $A$-module at the vertex $1$ (resp. $2$, $3$). Let $e=e_{1}+e_{2}$ be an idempotent.
Consider  the SMC reduction $\D^{\bb}(\mod A)/\thick(S_{3})$ of $\D^{\bb}(\mod A)$ with respect to $S_{3}$. It is equivalent to $\D^{\bb}(\mod B)$ by  Proposition \ref{Prop:SMCex} (1), where $B=eAe$ is given by the   quiver 
\xymatrix{1 \ar@/^0.6pc/[r]^{\alpha_{1}}  & 2 \ar@/^0.6pc/[l]^{\alpha_{2}}} with relations 
$\{ \alpha_{1}\alpha_{2}\alpha_{1}, \alpha_{2}\alpha_{1}\alpha_{2}\}$.

  Since $A$ is symmetric, then it is well-know that  $\D_{\sg}(A)\cong \un{\mod} A$ and the AR quiver of $\D_{\sg}(A)$ is given by $\Z A_{3}/\nu[1]$. In fact, we can describe it specifically as follows,

 {\small
       \begin{center}
         \begin{tikzpicture}[scale=0.6]
         \draw
         node (00) at (0,0) {$\begin{smallmatrix}&2&\\1& &3\end{smallmatrix}$}
         node (20) at (2,0) {$\begin{smallmatrix}2\end{smallmatrix}$}
         node (40) at (4,0) {$\begin{smallmatrix}1&&3\\&2 &\end{smallmatrix}$}
         node (60) at (6,0) {$\begin{smallmatrix}&2&\\1&&3\end{smallmatrix}$}
         node (80) at (8,0) {$\begin{smallmatrix}2\end{smallmatrix}$}
         node at (-2,0) {$\dots$}
         node at (9,0){$\dots$}
         
         node (31) at (3,1) {$\begin{smallmatrix}1\\2\end{smallmatrix}$}
         node (51) at (5,1) {$\begin{smallmatrix}3\end{smallmatrix}$}
         node (71) at (7,1) {$\begin{smallmatrix}2\\1\end{smallmatrix}$}
         node (11) at (1,1) {$\begin{smallmatrix}2\\1\end{smallmatrix}$}

         node (-11) at (-1,1) {$\begin{smallmatrix}3\end{smallmatrix}$}
         node (-1-1) at (-1,-1) {$\begin{smallmatrix}1\end{smallmatrix}$}

         node (7-1) at (7,-1) {$\begin{smallmatrix}2\\3\end{smallmatrix}$}
         node (1-1) at (1,-1) {$\begin{smallmatrix}2\\3\end{smallmatrix}$}
         node (3-1) at (3,-1) {$\begin{smallmatrix}3\\2\end{smallmatrix}$}
         node (5-1) at (5,-1) {$\begin{smallmatrix}1\end{smallmatrix}$}
         
           [dotted] (0.6,1.4) rectangle (6.6,-1.4);
          \fill[opacity=0.3, red]  (4.5,1.4)--(1.7,-1.4)--(3.2,-1.4)--(4.8,0.2)--(6.4,-1.4)--(8,-1.4)--(5.2,1.4)--(4.5,1.4) (0.4,-1.4)--(2,-1.4)--(-0.8,1.4)--(-2.4,1.4);          
                \end{tikzpicture} 
     \end{center}}    
\noindent where  the arrows are omitted and a fundamental domain is outlined in dotted line.      
By the definition of SMS reduction, we know that $$(\D^{\bb}_{\sg}(A))_{S_{3}}=\{ X\in \D^{\bb}_{\sg}(A)\mid \Hom_{\D^{\bb}_{\sg}(A)}(X[i], S_{3})=0=\Hom_{\D^{\bb}_{\sg}(A)}(S_{3}[i],X) \text{ with } i=0,1\}.$$
So the indecomposable objects of $(\D^{\bb}_{\sg}(A))_{S_{3}}$ are given by the AR quiver above without the shaded part. The AR quiver of $(\D^{\bb}_{\sg}(A))_{S_{3}}$ is $\Z A_{2}/\nu[1]$, which is the same as the AR quiver of $\D^{\bb}_{\sg}(B)$.
\end{Ex}

 \appendix
 
 \section{An equivalence induced by derived Schur functor}\label{appendix}
 Let $A$ be a non-positive proper dg algebra.  Let $e$ be an idempotent of $A$. Assume $e\in A^{0}$. Then $eA$ (resp. $Ae$) is a right (resp. left) dg $A$-module. 
We have a natural derived  Schur functor $F=?\otimes_{A}^{\bf L}Ae: \D(A)\ra \D(eAe)$, which restricts to a functor $F^{\bb}=?\otimes_{A}^{\bf L}Ae: \D^{\bb}(A)\ra \D^{\bb}(eAe)$.
It is well-known that $F$ admits a left adjoint $G=?\ot_{eAe}^{\bf L}eA$. We first give an easy observation.

\begin{Lem}\label{Lem:upperbounded}
Let $M\in \D^{\bb}(eAe)$. Then $G(M)\in \D(A)$ is upper bounded and $\h^{i}(G(M))$ is finite-dimensional for any $i\in \Z$.
\end{Lem}

\begin{proof}
Since $M\in \D^{\bb}(eAe)$, we may assume $M^{\gg 0}=0$.
We have
\[ D\h^{i}(G(M))=\Hom_{\D(A)}(G(M)[i], DA)=\Hom_{\D(eAe)}(M, F(DA)[i]). \] 
Since $M, F(DA)[i]\in \D^{\bb}(eAe)$ and by Lemma \ref{Lem:predg},  $\D^{\bb}(eAe)$ is Hom-finite, then $\h^{i}(M)$ is finite dimensional for any $i\in\Z$. Notice that $M$ and $eA$  are both upper bounded, so $G(M)=M\ot_{eAe}^{\bf L}eA$ is also upper bounded. 
\end{proof}

The following result should be well-known, but we could not find a reference. So we include a complete proof for the convenience of the reader.
\begin{Prop}\label{Prop:idem}
Let $A$ be a non-positive  proper dg algebra and $e\in A$ be an idempotent. Let $F, G$ be defined as above. Then $F$ induces a triangle equivalence $\overline{F}^{\bb}:\D^{\bb}(A)/\ker F^{\bb}\simeq \D^{\bb}(eAe)$.
\end{Prop}

\begin{Rem}
We point out that Proposition \ref{Prop:idem} is  known for finite-dimensional $k$-algebra (see \cite[Lemma 2.2]{Ch1}). But the approach in \cite{Ch1}  fails in  dg setting, so here we prove it  in a more direct way.
\end{Rem}

\begin{proof}
Notice that  $G$ is fully faithful (see for example, \cite[Lemma 4.2]{Keller}), then $\D(A)$ has a stable $t$-structure $(\im G, \ker F)$ and moreover, there is a triangle equivalence $\overline{F}: \D(A)/\ker F \simeq \D(eAe)$.
Considering the following commutative diagram.
\[ \xymatrix{ \D(A)/\ker F \ar[r]^{\overline{F}}_{\simeq} & \D(eAe) \\
\D^{\bb}(A)/\ker F^{\bb}  \ar[u]^{H}
\ar[r]^<<<<<{\overline{F}^{\bb}} & \D^{\bb}(eAe)\ar@{^{(}->}[u] }\]
where $H: \D^{\bb}(A)/\ker F^{\bb} \ra \D(A)/\ker F$ is the natural functor. To show $\overline{F}^{\bb}$ is fully faithful, it is enough to show $H$ is fully faithful and $\overline{F}^{\bb}$ is dense.

(1) $H$ is full.  Let $X, Y\in \D^{\bb}(A)$. Any morphism $X\ra Y$ in $\D(A)/\ker F$ can be written as $X \xleftarrow{s} Z\xrightarrow{f} Y$, such that there is a triangle 
\[ K \ra Z \xra{s} X \ra K[1]\]
with $K\in \ker F$.
In this case,  $F(Z)\cong F(X)$ in $\D(eAe)$ and thus $F(Z)\in \D^{\bb}(eAe)$.
Let $Z':=GF(Z)$. Then we have a natural triangle $Z'\xra{t} Z\ra K' \ra Z'[1]$ in $\D(A)$ given by the adjoint pair such that $K'\in \ker F$. It is easy to check that the morphism  $X \xleftarrow{st} Z' \xrightarrow{ft} Y$ is equivalent to $X \xleftarrow{s} Z\xrightarrow{f} Y$ in $\D(A)/\ker F$. By Lemma \ref{Lem:upperbounded}, we know that  $Z'$ is upper bounded and $\h^{n}(Z')$ is finite dimensional for any $n\in \Z$.

Now we consider the standard truncation of $Z'$. Since $X, Y\in \D^{\bb}(A)$, we can find small enough $m$ such that $$\Hom_{\D(A)}(\tau^{< m}Z', X)=0=\Hom_{\D(A)}(\tau^{<m}Z', Y).$$
Since $F(Z)=Ze$, which acts on cohomology, we may also assume $\tau^{{<m}}Z'\in \ker F$. 
Then we have the following diagram.
\[\small{ \xymatrix{ & \tau^{<m}Z' \ar[d] & \\
& Z' \ar[ld] \ar[rd] \ar[d]&\\
X & \tau^{\ge m}Z' \ar[l] \ar[r]& Y } }\] 
By our construction, $\tau^{\ge m}Z'\in \D^{\bb}(A)$ and the morphism $X\la \tau^{\ge m}Z'\ra Y$ is equivalent to $X\la Z' \ra Y$ in $\D(A)/\ker F$.
So the functor $H: \D^{\bb}(A)/\ker F^{\bb}\ra \D(A)/\ker F$ is full.

(2) $H$ is faithful.
Let $X\xleftarrow{p}U \xrightarrow{g} Y$ be any morphism in $\D^{\bb}(A)/\ker F^{\bb}$, which sends to zero map in $\D(A)/\ker F$. Then the morphism is equivalent to $X \xleftarrow{\rm Id} X \xrightarrow{0} Y$ in $\D(A)/\ker F$. So we have commutative diagram.
\[\small{\xymatrix{ & X \ar[ld]_{\rm Id} \ar[rd]^{0}& \\
X & W \ar[r]^{0} \ar[l]_{q} \ar[u]^{q} \ar[d]^{r}& Y\\
& U \ar[ur]_{g} \ar[ul]^{p} &}}\]
where $W\in \D(A)$, $\con(q)\in \ker F$,  $gr=0$ and $pr=q$.
By the same strategy in (1), we can take $W\in \D^{\bb}(A)$. Then $X\xleftarrow{p}U \xrightarrow{g} Y$ is also zero map in $\D^{\bb}(A)$. So $H$ is faithful.

(3) $\overline{F}^{\bb}$ is dense.
 Let $M\in \D^{\bb}(eAe)$.
We know  $G(M)\in \D(A)$ is upper bounded and $\h^{i}(G(M))$ is finite dimensional for any $i$ by Lemma \ref{Lem:upperbounded}.  Notice that we have 
\[ F(\tau^{\ge n}G(M))=(\tau^{\ge n}G(M))e=\tau^{\ge n}(G(M)e)=\tau^{\ge n}(FG(M))=\tau^{\ge n}(M).\]
Since $M\in \D^{\bb}(eAe)$, we may take  $n\ll 0$ such that $\tau^{\ge n}(M)\cong M$. Then $F(\tau^{\ge n}G(M))=M$ and $\tau^{\ge n}(G(M))\in \D^{\bb}(A)$. So $\widetilde{F}$ is dense.
\end{proof}


\end{document}